\documentclass[reqno,11pt]{amsart}
\usepackage{a4wide,xcolor,eucal,enumerate,mathrsfs}
\usepackage[normalem]{ulem}
\usepackage{amsmath,amssymb,epsfig,amsthm} %,bbm
\usepackage[latin1]{inputenc}
\usepackage{graphicx}

\usepackage{subfigure}
\numberwithin{equation}{section}
\numberwithin{figure}{section}

\newtheorem{theorem}{Theorem}[section]

\newtheorem{corollary}[theorem]{Corollary}
\newtheorem{lemma}[theorem]{Lemma}
\newtheorem{proposition}[theorem]{Proposition}

\theoremstyle{definition}
\newtheorem{definition}[theorem]{Definition}

\newtheorem{example}[theorem]{Example}

\newtheorem{question}[theorem]{Question}

\theoremstyle{remark}
\newtheorem{remark}[theorem]{Remark}

\def\eps{\epsilon}

\newcommand{\N}{\mathbb{N}}
\newcommand{\Q}{\mathbb{Q}}
\newcommand{\R}{\mathbb{R}}
\newcommand{\Z}{\mathbb{Z}}

\newcommand{\restr}[1]{\lower3pt\hbox{$|_{#1}$}}
\newcommand{\diam}{\mathop{\rm diam}\nolimits}
\newcommand{\inter}{\mathop{\rm int}\nolimits}
\newcommand{\cl}{\mathop{\rm cl}\nolimits}
\newcommand{\dist}{\mathop{\rm dist}\nolimits}

\newcommand{\distr}{\Delta}
\newcommand{\metr}{{\bf g}}

\newcommand{\VecM}{{\rm Vec}(M)}
\newcommand{\lam}{\lambda}
\newcommand{\del}{\delta}

\begin{document}

\title[Assouad dimension, Nagata dimension, and metric tangents]
{Assouad dimension, Nagata dimension, and uniformly close metric tangents}

\author{Enrico Le Donne}
\author{Tapio Rajala}
\address{Department of Mathematics and Statistics \\
         P.O. Box 35 (MaD) \\
         FI-40014 University of Jyv\"askyl\"a \\
         Finland}
\email{enrico.e.ledonne@jyu.fi}
\email{tapio.m.rajala@jyu.fi}

\thanks{T.R. acknowledges the support of the Academy of Finland project no. 137528.}
\renewcommand{\subjclassname}{%
 \textup{2010} Mathematics Subject Classification}
\date{December 25, 2013}

\begin{abstract}
We study the Assouad dimension and the Nagata dimension of metric spaces.
As a general result, we prove that the Nagata dimension of a metric space is always bounded from above by the Assouad dimension.
Most of the paper is devoted to the study of when 
 these metric dimensions of a metric space are locally given by the dimensions of its metric tangents. 
Having uniformly close tangents is not sufficient.
  What is needed in addition is either 
that the tangents have dimension with uniform constants independent from the point and the tangent, or that the tangents are unique.
We will apply our results to equiregular subRiemannian manifolds and show that locally their Nagata dimension equals the topological dimension.
\end{abstract}

\maketitle
\tableofcontents
\newpage
\section{Introduction}

Assouad dimension, Nagata dimension, and metric tangents are relevant in the program of doing analysis in the metric space setting.
The Assouad dimension is a quantification of the doubling property. It plays a role, for example, in the study of spaces that are
quasisymmetrically embeddable in Euclidean spaces, of fractal sets, and of boundaries of groups,
\cite{Assouad79,Assouad83,Luukkainen,Keith-Laakso,Mackay}. 
The Nagata dimension, introduced in \cite{Nagata58,Assouad82}, is a local-and-global metric version of the topological dimension.
The bounded-scale version is called linearly controlled dimension and the large-scale
version is called asymptotic Assouad-Nagata dimension. Such dimensions are relevant for embeddings,
in particular in Geometric Group Theory, \cite{Buyalo-Dranishnikov-Schroeder}.
Nagata dimension links with quasisymmetric embedding into metric trees and with Lipschitz
extension properties, \cite{Lang-Schlichenmaier,Wenger-Young}.

In \cite{Lang-Schlichenmaier} it has been shown that doubling metric spaces have finite Nagata dimension.
In this paper, we prove the sharp bound:
\begin{theorem}\label{thm:nagaasso}
 For all metric spaces $X$, the Nagata dimension of $X$ is less than or equal to the Assouad dimension of $X$.
 %The inequality $\dim_NX \le \dim_AX$ holds for all metric spaces $X$.
\end{theorem}

Metric tangents provide a way of studying the infinitesimal properties of metric spaces.
Limits of metric spaces were introduced by Gromov in the setting of Geometric Group Theory to
study asymptotic cones of groups of polynomial growth \cite{Gromov-polygrowth}.
However, many of the results hold in the context
of metric spaces with finite Assouad dimension
and have applications in different areas of mathematics, for instance in the study of
limits of Riemannian manifolds with curvature bounds and Reifenberg-flat metric spaces,
see \cite{Cheeger-Colding}, \cite{David-Toro}, and subsequent work.
Additional results regarding tangents of general metric spaces can be found in \cite{Hanson-Heinonen,LeDonne6,Herron_tan}.

In this paper we study how and when one can deduce the Nagata dimension or the Assouad dimension 
of a space, knowing the respective dimensions of its tangents.
A concrete application of our results is given by the Lipschitz extension problem for
subRiemannian manifolds.
Lang and Schlichenmaier gave a connection of the Lipschitz extension problem with the Nagata dimension and
the Lipschitz connection property.
Lipschitz connectivity and 
 Lipschitz homotopy groups have recently been studied  in
\cite{DeJarnette-Hajlasz-Lukyanenko-Tyson,Wenger-Young,Hajlasz-Schikorra}.
By the results in this paper, knowing now the Nagata dimension of equiregular subRiemannian manifolds,
one can deduce for example that a partially defined Lipschitz map $f \colon A   \to Y$
from a subset $A$ of an equiregular subRiemannian manifold $M$ can be Lipschitz extended on compact sets of
$M$  if $Y$ is Lipschitz $m$-connected for all $m$ strictly smaller than the topological
dimension of $M$.
For such applications see \cite[Theorem 1.5 and  Theorem 1.6]{Lang-Schlichenmaier}.
Another property one  deduces for a space $(X,d)$ with Nagata dimension at most some number $n \in \N$
is that for sufficiently small $p \in (0,1)$ there exists a bi-Lipschitz embedding of $(X,d^p)$
 into the product of $n + 1$ metric trees, see \cite[Theorem 1.3]{Lang-Schlichenmaier}.
\\

Let us now present in detail our results on tangent spaces.
Let $X$ be a metric space. 
%\begin{definition}
For each $x\in X$, let ${\rm Tan}(X,x)$ be the collection of all the metric spaces tangent to $X$ at $x$, in the pointed Gromov-Hausdorff sense.
We say that   $X$ has {\em uniformly close tangents} if, for each $x\in X$, the convergence of the dilated spaces of $(X,x)$ toward  ${\rm Tan}(X,x)$ is uniform. In other words and more generally,
we say that on a subset $K\subset X$  the {\em convergence to tangents is uniform} if,
   %there is at least a tangent $Y$ of $X$ at $x$, and, 
 for all $
\eps>0$,  there exists $\lambda_\eps>0$ such that, for all $k\in K$ and all $\lambda>\lambda_\eps$, there exists 
 a tangent $Y$ of $X$ at $k$ with
$${\rm Dist}_{GH}( (\lambda X,k), Y) <\eps.$$
Here ${\rm Dist}_{GH}$ is a specific distance that we fix in Section \ref{sec:GH} to induce the Gromov-Hausdorff topology on pointed metric spaces and
$\lambda X$ is the metric space $(X, \lambda d_X).$
%\end{definition}
The condition of uniform convergence to tangents is motivated by the fact that
this is what happens on equiregular subRiemannian manifolds, see Theorem \ref{Mitchell:thm}.

Assuming uniform convergence towards unique tangents, our first result for  the Nagata dimension, which we denote by $\dim_N$,   is the following:
\begin{theorem}\label{teorema1}
Let $X$ be a metric space that at every point admits a single tangent space.
%Assume that the convergence toward the tangents is uniform on compact sets (as just defined).
Let $Y \subseteq X$ be a relatively compact set with $\dim_NY < \infty$.
Assume that the convergence toward the tangents is uniform on the closure of $Y$ (as just defined).
Then we have
\[
 \sup_{x \in \inter Y} \dim_N T_xX \le \dim_NY \le \sup_{x \in \cl{Y}} \dim_N T_xX.
\]
% the Nagata dimension of any nonempty relatively compact open set $Y\subseteq X$ of finite Nagata dimension equals the supremum of the Nagata dimensions of the
%tangents of $X$ at points on the closure of $Y$.
\end{theorem}

Here $\inter Y$ and $\cl{Y}$ denote the interior and the closure of the set $Y$, respectively.
From Example \ref{ex:nagata} we see that the assumption on uniqueness of tangents is necessary in Theorem \ref{teorema1}.
Relative compactness of $Y$ is needed already to handle the large scales, and the necessity of the interior of $Y$
is seen by taking $Y = \{0\}$ and $X = \R$.
We shall prove Theorem~\ref{teorema1} at the end of Section \ref{Sec:Nagata_uniform_tangents} as a consequence of 
Theorem~\ref{thm:uniformtangents}, with assumption $(ii)$.
An application of Theorem~\ref{teorema1} is the following result, which was actually our initial goal
(the notion of equiregular subRiemannian manifold is recalled in  Section \ref{sec:equiSR}).

\begin{corollary}\label{corollarioSRM}
 Let  $(M,d_{cc})$ be an equiregular subRiemannian manifold. %Let $p \in M$ be a regular point for the horizontal   distribution.
  Then the  Nagata   dimension  of any open bounded nonempty subset of $M$ equals   the topological dimension of the manifold.
\end{corollary}

Corollary \ref{corollarioSRM} relies also on a result by Urs Lang and the first-named author which states that the Nagata dimension of a Carnot group equals its topological dimension. For completeness, we include a short proof of this fact in Section \ref{sec:equiSR}.
\\

We prove the following analog of Theorem~\ref{teorema1} for the Assouad dimension, which we denote by $\dim_A$.
\begin{theorem}\label{teorema1b}
Let $X$ be a metric space that at every point admits a single tangent space.
%Assume that the convergence toward the tangents is uniform on compact sets (as defined above).
Let $Y \subseteq X$ be a relatively compact set.
Assume that the convergence toward the tangents is uniform on $Y$ (as defined above).
Then we have
\[
 \sup_{x \in \inter Y} \dim_A T_xX \le \dim_AY \le \sup_{x \in \cl{Y}} \dim_A T_xX.
\]
\end{theorem}

Theorem~\ref{teorema1b} is a consequence of Theorem~\ref{thm:assouad2}.
These results will be proved in Section \ref{Sec:Assouad_uniform_tangents}. 
Theorem  
\ref{teorema1b} gives an alternative proof of a fact essentially proven in
\cite{nagelstwe}.
Namely, the Assouad dimension  of an equiregular subRiemannian manifold equals the
Assouad dimension of its tangents.

One direction of research where the above results can be used
is the study of generalizations of Reifenberg 
vanishing-flat metric spaces where the model space $\R^n$ is replaced by any fixed doubling metric space.
Namely, we say that a metric space $X$ is {\em vanishing-flat modeled on a metric space $Y$}
if $Y$ is the tangent space at any point $x$ of $X$ and the convergence toward the tangents is uniform.
The case when $Y$ is a Euclidean space is called 
Reifenberg vanishing-flat metric spaces, and it has been manly considered in \cite{Cheeger-Colding,David-Toro}.
It is a natural problem to study what properties of $X$ can then be deduced from the ones of $Y$.
%We have  more 
%general applications to  Reifenberg 
%vanishing-flat metric spaces where the model space $\R^n$ is replaces with any fixed doubling metric space.
%Namely, we say that a metric space $X$ is {\em vanishing-flat modeled} on a metric space $Y$
%if $Y$ is the tangent space  at any point $x$ of $ X$ and the convergence toward the tangents is uniform.
%The case when $Y$ is a Euclidean space is called 
%Reifenberg vanishing-flat metric spaces, and it has been manly  considered in
% \cite{Cheeger-Colding} and \cite{David-Toro}.
%It is a natural problem to study what properties of $X$ can then be deduced from the ones of $Y$.
\\

We shall also provide results when the tangents are not assumed to be single spaces. However, without such an assumption, it is necessary to require uniformity in values of the constants appearing in the definition of the dimension of the tangents.
Such uniformity of constants is true for instance for many self-similar spaces and these spaces usually have more than one tangent at every point.
Let us briefly recall that the linearly controlled dimension is the bounded-scale version of the Nagata dimension.
\begin{theorem}\label{teorema2}
Let $X$ be a metric space of finite linearly controlled dimension.
Assume that $X$ has uniformly close tangents (as defined above).
Then the linearly controlled dimension of $X$ equals the infimum of all integers $n$ for which, for all $x\in X$,  the linearly controlled dimension of any $Y \in    {\rm Tan}(X,x)$ is  at most $  n$ with constants independent from $Y$   and $x$.
%unif constant - Nagata
\end{theorem}
See Theorem~\ref{thm:uniformtangents}, with assumption $(i)$, for a more explanatory statement of the upper bound.
The lower bound follows from Corollary \ref{cor:Nagatalower}. 
%{\color{blue} 
%Alternatively, one can also deduce Corollary \ref{corollarioSRM} from Theorem \ref{teorema2}. Indeed, one can prove that  in equiregular subRiemannian manifolds the constants involved in the Nagata dimension of the tangents change continuously, hence they are uniform on compact sets.
%}
\\
 
We also have the analogue 
of Theorem~\ref{teorema2}
for the Assouad dimension.
\begin{theorem}\label{teorema3}
%unif constant - Assouad
Let $X$ be a metric space.
Assume that $X$ has uniformly close tangents (as defined above).
%For each $x\in X$, let ${\rm Tan}(X,x)$ be the collection of all the metric spaces tangent to $X$ at $x$.
%Assume that, for each $x\in X$, the convergence of the dilated spaces of $(X,x)$ toward  ${\rm Tan}(X,x)$ is uniform.
Then the Assouad dimension of $X$ equals the infimum of all  $\alpha\geq 0 $ for which, for all $x\in X$,  the Assouad dimension of any $Y \in    {\rm Tan}(X,x)$ is  at most   $ \alpha$ with constants independent from $Y$   and $x$.
\end{theorem}
A more explanatory statement of the upper bound is given in Theorem~\ref{thm:assouad1}, which will be an immediate consequence of
Proposition \ref{prop:assouad1}.
The lower bound is given by Corollary \ref{cor:Assouadlower}.
\\

The paper is organized as follows. Section \ref{sec:preli} is devoted to preliminaries.
We recall the definitions and basic properties of Assouad dimension, Nagata dimension, and locally controlled dimension.
In Section \ref{sec:GH}, we give the definition of Gromov-Hausdorff distance for pointed metric spaces
and we define the set of tangents.
We provide some remarks about the lower semicontinuity of Assouad dimension and Nagata dimension.
In particular, Corollary~\ref{cor:Assouadlower} (resp. Corollary \ref{cor:Nagatalower})
gives the lower bound for Theorem~\ref{teorema1b} and Theorem~\ref{teorema3} (resp. Theorem~\ref{teorema1} and Theorem~\ref{teorema2}).
In Example \ref{ex:nonuniformtangent}, we show that in general, even for compact subsets of $\R$, the dimension of all the tangents
could be strictly smaller than the dimension of the set. However, in Proposition \ref{prop:weaktangentnagata} we prove that, if
a doubling space has Nagata dimension equal to one, it has some weak tangent with Nagata dimension one.

In Section \ref{sec:uniform} we study metric spaces with uniformly close tangents.
In Example \ref{ex:nagata} and Example~\ref{ex:assouad} we show that having uniformly close tangents does not imply an upper
bound for the dimension of the space in terms of
 the dimensions
 the tangents.
In Theorem~\ref{thm:uniformtangents} we provide such a bound with the additional assumption that either the tangents
have linearly controlled dimension less than $n$ with respect to a uniform constant $c$, or the tangents are unique.
Such theorem gives the missing upper bound for Theorem~\ref{teorema2} and Theorem~\ref{teorema1}.
In Section \ref{Sec:Assouad_uniform_tangents} we consider Assouad dimension. We prove Theorem~\ref{thm:assouad1}
(resp. Theorem~\ref{thm:assouad2}) giving the upper bound needed for concluding the proof of
Theorem~\ref{teorema3} (resp. Theorem~\ref{teorema1b}).

In Section \ref{sec:equiSR} we apply Theorem~\ref{teorema1} to prove Corollary \ref{corollarioSRM},
after recalling some results on subRiemannian geometry and Carnot groups.
In Section \ref{Sec:nagaasso} we prove Theorem~\ref{thm:nagaasso}
and Theorem~\ref{thm:nagaasso2},
which is a bounded-scale version of Theorem~\ref{thm:nagaasso}. %the previous theorem.

\section{Preliminaries}\label{sec:preli}
%
%Define
%\[
% \delta(\lambda,x) := d_{GH}(    (X,\lambda d_X,x),T_xX)
%\]
%for all $x \in X$ and $\lambda > 0$.

\subsection{Assouad dimension and Nagata dimension}
In this paper we will consider the notion of Assouad dimension. 
 %Other  notion of metric dimension  that we will consider is  the  Assouad dimension. 
It is also known with other names such as: metric covering dimension, uniform metric dimension, or doubling dimension.
We recall here the definition from  \cite[page 81]{Heinonenbook}. %Assouad dimension here.
 The {\em Assouad dimension} of a metric space $X$ is denoted by  $\dim_{A}X$ and is defined as the infimum of all numbers $\beta > 0$ with the property that there exists some $C>1$ such that, for every $\eps>0$,
 every set of diameter $D$ can be covered by using no more than
 $C\eps^{-\beta}$ sets of diameter at most $\eps D$.  
 In this case, we say that the Assouad dimension is less than or equal to $\beta$ with constant $C$.
 
% \begin{equation}\label{Ass Dim}
%\begin{array}{l}
%\text{there exists a constant $C'=C'_\beta>1$  such that, for all $0<r<R<R'$, 
% }\\
%\text{any ball of radius $R$ in $X$  }\\
%\text{
%can be covered with  less than $C'\left( {R}/{r}\right)^\beta$ balls of radius $r$ in $X$.}
%\end{array}
%\end{equation}
We will need a quantified and local version of the definition that makes explicit the constants involved.
\begin{definition}[Assouad dimension up to a scale]
Let $\bar R>0$ and $C>1$.
We say that a metric space $X$ has {\em Assouad dimension at most} $\beta$  {\em up to scale} $\bar R$  {\em with constant} $C$ if,
for all $0<r<R< \bar R$, any ball of radius $R$ in $X$ can be covered with %no more than
 %{\color{}needs at most} 
 $C\left( {R}/{r}\right)^\beta$
 or less balls of radius $r$ in $X$. % {\color{}to cover it}.
 In this case we write 
 $\dim_A(X,C,\bar R)\leq 
\beta$.
 \end{definition}

% As pointed out in \cite[Ex.~10.17]{Heinonenbook},
%the Assouad dimension  can be defined
%equivalently as the infimum of all numbers $\beta > 0$ with the property that, for some $C> 1$,  for every $\eps>0$, every ball
%of radius $r > 0$ has at most $C\eps^{-\beta}$ disjoint points of mutual distance at least $\eps r$. %, for
% %independent of the ball.
 
Metric spaces with finite Assouad dimension are precisely the doubling metric spaces. We recall that a metric space is {\em  doubling with constant }$L$,
 for some  $L>0$, if, for every $s > 0$, 
every subset of the metric space with diameter at most $ 2s$ can be covered by $L$ or fewer 
sets of diameter at most $ s$.   
\\

Other  notion of metric dimension  that we will consider is  the  Nagata dimension. 
 %In this paper we will consider the notion of Nagata dimension. 
 Before giving the definition, let us   recall some basic terminology, following \cite{Lang-Schlichenmaier}. 
Two subsets $A,B$ of a metric space are \emph{$s$-separated}, for some constant $s\geq0$, if
$\dist(A, B) := \inf\{d(a,b)\;;\;a\in A,b\in B\} \geq s$.
A family  of subsets is called \emph{$s$-separated} if 
each distinct pair of element in  it is $s$-separated.
Let  $\mathcal{B} $ be a cover of a metric space $X$. Then, for $s > 0$, the 
\emph{$s$-multiplicity} of $\mathcal{B}$ is the infimum of all $n$ such that 
every subset of $X$ with diameter at most $ s$ meets at most $n$ members of the 
family  $\mathcal{B} $. 
Furthermore,
$\mathcal{B}$ is called \emph{$D$-bounded}, for some constant $D \ge 0$, 
if $\diam B := \sup\{d(x,x')\;;\;x,x'\in B\} \le D$,
for all $B \in \mathcal{B}$.

\begin{definition}[Nagata dimension]\label{def nag dim}
 Let $X$ be a metric space. The {\em Nagata dimension}, or Assouad-Nagata dimension,
 of $X$ is denoted by $\dim_N X$ and is defined as the infimum of all integers $n$ with the following property: there
exists a constant $c > 0$ such that, for all $s > 0$, $X$ admits a $cs$-bounded cover with $s$-multiplicity
at most $n + 1$.
\end{definition} 

 As shown in \cite[Proposition 2.5]{Lang-Schlichenmaier},
 the Nagata dimension  can be defined
equivalently as  the infimum of all integers $n$ with the following property: 
\begin{equation}\label{Nag Dim}
\begin{array}{l}
\text{ there exists a constant $c> 0$ such that for all $s > 0$, the metric space admits}\\
\text{ 
an $s$-bounded cover of the form $\mathcal{B} = \bigcup_{k=0}^n \mathcal{B}_k$ where each   $\mathcal{B}_k$ is $cs$-separated.}
\end{array}
\end{equation}
The constant $c$ in \eqref{Nag Dim} can be different from the constant $c$ in Definition \ref{def nag dim}.
\\

We should notice that the notion of Nagata dimension is global. Moreover, it is   both for small and large scales.
We will need the bounded-scale version. For a better formulation of our statements, we give a definition that points out both the scale and the constant. 
\begin{definition}[Nagata dimension at a scale]\label{def:Nag_scale}
 We say that a metric space $X$ has Nagata dimension bounded from above by $n \in \N:=\{0, 1, 2, \ldots\}$ with constant $c>0$ at scale $s>0$
 and write
 \[
  \dim_N(X,c,s) \le n,
 \]
 if
  the metric  space admits 
an $s$-bounded cover of the form $\mathcal{B} = \bigcup_{k=0}^n \mathcal{B}_k$ 
where 
each    $\mathcal{B}_k$ is $cs$-separated.
 \end{definition}
 
% if there exist $n + 1$ collections $\{U_i^0\}_{i \in I_0}, \dots, \{U_i^n\}_{i \in I_n}$ of subsets of $X$ with the properties
% \[
%  \diam(U_i^j) \le s \qquad \text{for all }j = 0, \dots, n \text{ and }i \in I_j,
% \]
% \[
% \dist(U_i^j, U_k^j) \ge cs\qquad   \text{for all }j = 0, \dots, n \text{ and }i,k \in I_j \text{ with }i \ne k
% \]
% and
% \[
%  X = \bigcup_{j = 0}^n\bigcup_{i \in I_j}U_i^j.
% \]
\begin{definition}[Linearly controlled dimension]\label{lin con dim}
We say that a metric space $X$ has {\em linearly controlled dimension} bounded from above by $n$, written as
 \[
  \dim_{LC} X \le n,
 \]
 provided that there exists $c >0$ and $s_0 > 0$ such that $\dim_N(X,c,s) \le n$, for all $0 < s < s_0$.
\end{definition}

The notation $\ell$-dim is also used for the linearly controlled dimension, see \cite[Section~9.1.4]{Buyalo_Schroeder_book}

Notice that the space $X$ has Nagata dimension bounded from above by $n$ 
exactly when we can take $s_0 = \infty$, in other words if there exists $c >0$ such that $\dim_N(X,c,s) \le n$ for all $s > 0$.

Since any  cover for a space $X$ restricts to a cover to any of its subsets $Y$, we have the easy inequality
 $\dim_NY \le \dim_NX$, for all $Y \subset X$.
 Moreover, if $Y \subset X$, then 
  \begin{equation}\label{lma:subset}
  \dim_N(X,c,s) \le n \implies    \dim_N(Y,c,s) \le n.
\end{equation}
%\begin{proof}
% Any suitable cover for $X$ works also for $Y$.
%\end{proof}

\subsection{Basic facts about the dimensions}

Here is a first lemma showing that, in the local version of the Assouad dimension,  we have no problem if the radius a priori depends on the exponent.
\begin{lemma}\label{lem_ass_2}
%Let $X$ be a metric space and $\alpha\geq0$.
%Assume that, for all $\beta>\alpha$, there exist two constants $C_\beta>1$ and $ R_\beta>0$ such that, for all $0<r<R< R_\beta$, any ball of radius $R$ in $X$ can be covered with  less than  $C_\beta\left( {R}/{r}\right)^\beta$ balls of radius $r$ in $X$.
%Then there exists $R'>0$ such that, for all $\beta>\alpha$, there exists a constant $C'_\beta>1$  such that, for all $0<r<R<R'$, any ball of radius $R$ in $X$ can be covered with  less than $C'_\beta\left( {R}/{r}\right)^\beta$ balls of radius $r$ in $X$.

Let $X$ be a metric space and $\alpha\geq0$.
Assume that, for all $\beta>\alpha$, there exist two constants $C_\beta>1$ and $ R_\beta>0$ such that,
 $\dim_A(X,C_\beta,  R_\beta)\leq 
\beta$. 
%\begin{equation*}%\label{Ass Dim}
%\begin{array}{l}
%\text{ for all $0<r<R< R_\beta$, any ball of radius $R$ in $X$ can be covered
% }\\
%\text{with  less than  $C_\beta\left( {R}/{r}\right)^\beta$ balls of radius $r$ in $X$.}
%\end{array}
%\end{equation*}
Then there exists $R'>0$ such that, for all $\beta>\alpha$, there exists a constant $C'_\beta>1$  such that, 
 $\dim_A(X,C'_\beta,  R')\leq 
\beta$.
%
%\begin{equation*}%\label{Ass Dim}
%\begin{array}{l}
%\text{for all $0<r<R<R'$, any ball of radius $R$ in $X$ can be covered 
% }\\
%\text{with  less than $C'_\beta\left( {R}/{r}\right)^\beta$ balls of radius $r$ in $X$.}
%\end{array}
%\end{equation*}
 \end{lemma}
 
\begin{proof}
 Fix $\eta>\alpha$.
 We can take $R':=  R_{\eta}$.
Indeed, pick any  $\beta>\alpha$.
 Fix $0<r<R<  R_{\eta}$ and $x\in X$.
 If $R<  R_\beta$, we are done. So we assume $R\geq   R_\beta$.
 Cover $B(x ,   R_{\eta})$ with % at most
$C_{\eta}\left( {  R_{\eta}}/{  R_{\beta}}\right)^{\eta}$ or less balls of radius $  R_{\beta}$.
Cover each of these balls with at most
$C_{\beta}\left( {  R_{\beta}}/{r}\right)^\beta$ balls of radius $r$.
Hence, the ball  $B(x ,   R  )$, which is inside  $B(x ,   R_{\eta}),$ needs no more than
$C_{\eta}\left( {  R_{\eta}}/{  R_{\beta}}\right)^{\eta} C_{\beta}\left( {  R_{\beta}}/{r}\right)^\beta$ balls of radius $r$ to cover it.
The proof is concluded by putting
$C'_{\beta}:=
C_{\eta}
C_{\beta}
\left( {  R_{\eta}}/{  R_{\beta}}\right)^{\eta} $,
since then
$C'_{\beta} 
\left( {  R_{\beta}}/{r}\right)^\beta
<
C'_{\beta} 
\left( {  R }/{r}\right)^\beta.$
\end{proof}

Regarding the Nagata dimension,
we shall study what happens when we consider spaces that are the union of spaces of which we know their Nagata dimension.
We start with an easy statement, whose proof is straightforward.
\begin{lemma}\label{lma:dimforseparated}
 Let $X = \bigcup_{i \in I} X_i$ with $X_i$ that are $r$-separated and 
 $
  \dim_N(X_i,c,s) \le n$, for all $i \in I$.
% and
% \[
%  \dist(X_i,X_j)\ge r \qquad \text{for all }i,j\in I \text{ with } i\ne j.
% \]
 Then
 \[
  \dim_N\left(X ,\min\left\{c,\frac{r}{s}\right\} , s\right) \le n.
 \]
\end{lemma}

Regarding finite unions, we show the following Lemma~\ref{lma:twounion}, which  is essentially the quantified version of \cite[Theorem 2.7]{Lang-Schlichenmaier}.
\begin{lemma}\label{lma:twounion}
 Let $X,Y$ be subsets of a metric space and $c_1,c_2 \in (0,1)$ and $0 \le s_0 < s_1 < \infty$ such that
 \[
  \dim_N(X,c_1,s), \dim_N(Y,c_2,s) \le n ,\qquad \text{for all }s_0 \le s \le s_1.
 \]
 Then
\[
  \dim_N\left(X \cup Y,\frac{c_1c_2}{5},s\right) \le n, \qquad \text{for all }\left(2 + \frac3{c_1}\right)s_0 \le s \le \left(1 + \frac{2}{3}c_1\right)s_1.
 \]
\end{lemma}
\begin{proof}
 Let $s \in [s_0,s_1]$ and take 
 an $s$-bounded cover $\bigcup_{j=0}^n\{U_i^j\}_{i \in I_j}$ of $X$ where    $\{U_i^j\}_{i \in I_j}$
 is $c_1s$-separated, for all $j$.
 %a cover $\{U_i^0\}_{i \in I_0}, \dots, \{U_i^n\}_{i \in I_n}$ of $X$ with the properties
 %\[
 % \diam(U_i^j) \le s \qquad \text{for all }j = 0, \dots, n \text{ and }i \in I_j,
 %\]
 %\[
 %\dist(U_i^j, U_k^j) \ge c_1s\qquad   \text{for all }j = 0, \dots, n \text{ and }i,k \in I_j \text{ with }i \ne k.
 %\]
 Assuming $\frac{1}3c_1s \ge s_0$, take also
 a $\frac{1}3c_1s$-bounded cover $\bigcup_{j=0}^n\{V_i^j\}_{i \in J_j}$ of $Y$ where, for each $j$, the family $\{V_i^j\}_{i \in J_j}$
 is $\frac13c_1c_2s$-separated.
 %a cover $\{V_i^0\}_{i \in J_0}, \dots, \{V_i^n\}_{i \in J_n}$ of $Y$ with the properties
 %\[
 % \diam(V_i^j) \le \frac{1}3c_1s \qquad \text{for all }j = 0, \dots, n \text{ and }i \in J_j,
 %\]
 %\[
 %\dist(V_i^j, V_k^j) \ge \frac13c_1c_2s\qquad   \text{for all }j = 0, \dots, n \text{ and }i,k \in J_j \text{ with }i \ne k.
 %\]
For each $j = 0, \dots, n$, define
\[
K_j = \left\{i \in J_j \,:\, \dist\left(V_i^j , \bigcup_{k \in I_j}U_k^j\right) \ge \frac13c_1s\right\}.
\]
Using $K_j$ write new collections of sets covering $X \cup Y$ as follows. For all $i \in I_j$ and $j = 0, \dots, n$
set 
\[
 W_{(1,i)}^j := U_i^j \cup \bigcup_{\dist\left(V_k^j , U_i^j\right) < \frac13c_1s} V_k^j
\]
and for all $i \in K_j$, $j = 0, \dots, n$
\[
 W_{(2,i)}^j := V_i^j.
\]
Abbreviate $L_j = (\{1\} \times I_j) \cup (\{2\} \times K_j)$. It is easy to check that 
 $\bigcup_{j=0}^n\{W_i^j\}_{i \in L_j}$ is a $\left(1 + \frac23c_1\right)s$-bounded cover
of $X \cup Y$ where each   $\{W_i^j\}_{i \in L_j}$ is $\frac13c_1c_2s$-separated.
%the collections $\{W_i^0\}_{i \in L_0}, \dots, \{W_i^n\}_{i \in L_n}$
%now have the properties
%\[
%  \diam(W_i^j) \le \left(1 + \frac23c_1\right)s \qquad \text{for all }j = 0, \dots, n \text{ and }i \in L_j,
% \]
% \[
% \dist(W_i^j, W_k^j) \ge \frac13c_1c_2s\qquad   \text{for all }j = 0, \dots, n \text{ and }i,k \in L_j \text{ with }i \ne k
% \]
% and
% \[
%  X \cup Y = \bigcup_{j = 0}^n\bigcup_{i \in L_j}W_i^j.
% \]
 Now
 \[
  \frac{\dist(W_i^j, W_k^j)}{\diam(W_i^j)} \ge \frac{\frac13c_1c_2s}{\left(1 + \frac23c_1\right)s} \ge \frac{c_1c_2}{5}.
 \]
 Therefore
 \[
  \dim_N\left(X \cup Y,\frac{c_1c_2}{5},\left(1 + \frac23c_1\right)s\right) \le n,
 \]
 under the assumptions that we have made for $s$, namely,
 \[
  \frac{3s_0}{c_1}\le s \le s_1.
 \]
 The claim then follows.
\end{proof}

Iterating the above lemma we get the analogous statement for finite unions.
\begin{corollary}\label{cor:Npart}
 Assume that there exist $c \in (0,1)$, $s_0,s_1>0$ and $n, N \in \N$ such that
 \[
  \dim_N(X_i,c,s) \le n, \qquad \text{for all }s_0 \le s \le s_1 \text{ and }i = 1, \dots, N.
 \]
 Then
 \[
  \dim_N\left(\bigcup_{i=1}^NX_i,\frac{c^N}{5^{N-1}},s\right) \le n, \qquad \text{for all }\left(2 + \frac3{c}\right)^{N-1}s_0 \le s \le s_1.
 \]
\end{corollary}

By the above results (or just by  \cite[Theorem 2.7]{Lang-Schlichenmaier})   finite unions of   sets with Nagata dimension $n$  have Nagata dimension $n$.
The same is not valid for a countable union. Indeed, the space $\Z$
has dimension one but it is the countable union of points, which have dimension zero.
However, we can conclude that the Nagata dimension of a space is $n$, if we know that 
each ball $B(x, r)$, at a fixed point $x\in X$,   has Nagata dimension equal to $n$ for the same constant $c$. 
In the case of separable metric spaces,  we have the following more general fact. 
Suppose that there exists an increasing sequence of
 subsets $X_1 \subset X_2 \subset \cdots  \subset X$ with $\bigcup_{i=1}^\infty X_i = X$ separable and such that all
 $X_i$ have Nagata dimension equal to $n$ with the same constant $c>0$ in \eqref{Nag Dim}. Then one can show that $\dim_NX = n$.
Without the assumption of separability, we can show the following fact, which is the one that we will use to give an upper bound for the Nagata dimension of limits of metric spaces.
 % need the following weaker result.

\begin{lemma}\label{union}
Let $X$ be a metric space. Let $n\in\N$, $0 < s_1 < \infty$, $0<c<1$ and $x_0 \in X$.
%Assume there is  a sequence $X_m\subset X$ of increasing subsets 
%$$X_1\subseteq X_2 \subseteq\ldots \subseteq X_m\subseteq\ldots, $$
% with Nagata dimension equal to $n$ for the same constant $c$ and 
%$X=\cup_{m\in\N} X_m$. 
Assume there is  a sequence of radii $r_m\to\infty$ such that
\[
 \dim_N(B(x_0,r_m),c,s) \le n, \qquad \text{for all }0 < s \le s_1 \text{ and all }m = 1,2,\dots.
\]
Then
\[
 \dim_N\left(X,\frac{c^2}5,s\right) \le n, \qquad \text{for all } 0 < s \le \left(1+\frac23c\right)s_1.
\]
%each ball $B(x_0,r_m)$ has Nagata dimension equal to $n$ for the same constant $c$. 
%Then  the Nagata dimension of $X$ is $n$.
\end{lemma}
\begin{proof}
%If $S = \infty$, take $s_1 > 0$. Otherwise take $s_1 = S$. 
By taking a subsequence if necessary, we may assume that $r_{m+1}-r_m > cs_1$.
Let
\[
A_k:=B(x_0,r_{2k})\setminus B(x_0,r_{2k-1}), \qquad
Y:= \bigcup_{k=1}^\infty A_k,\qquad \text{ and }\qquad Z:=X\setminus Y. 
\]
Notice that
\[
 Z = B(x_0,r_1)\cup \bigcup_{k=1}^\infty B(x_0,r_{2k+1})\setminus B(x_0,r_{2k}).
\]

For all $0 < s \le s_1$, we have $\dim_N(A_k,c,s) \le \dim_N(B(x_0,r_{2k}),c,s) \le n$ and the family $\{A_k\}$
 is $cs_1$-separated. Therefore, by Lemma~\ref{lma:dimforseparated}, 
\[
 \dim_N(Y,c,s) \le n, \qquad \text{for all } 0 < s \le s_1.
\]
Similarly,
\[
 \dim_N(Z,c,s) \le n, \qquad \text{for all } 0 < s \le s_1.
\]
Therefore the claim follows from Lemma~\ref{lma:twounion}.
%\[
% \dim_N\left(X,\frac{c^2}5,s\right) \le n, \qquad \text{for all } 0 < s \le \left(1+\frac23c\right)s_1.
%\]
\end{proof}

%\proof
%Fix $s > 0$.
%We may assume that $r_{m+1}-r_m > s ,$ (or maybe $>2s$?) for all $m\in\N$. The idea is to write the space $X$ as a disjoint union of two subsets $Y$ and $Z$ made of disjoint annuli.
%Let 
%$$A_k:=B(x_0,r_{2k})\setminus B(x_0,r_{2k-1}), \qquad
%Y:= \bigcup_{k=1}^\infty A_k,\qquad \text{ and }\qquad Z:=X\setminus Y.$$
%
%We claim that $Y$ and $Z$ have Nagata dimension equal to $n$. Such a claim would finish the proof since Theorem 2.7 in \cite{Lang-Schlichenmaier}.

%We will only show that $Y$  has  dimension  $n$, since the proof for $Z$ is similar.
%Since each $A_k$ is a subset of $B(x_0,r_{2k})$, then it has  dimension  $n$, with constant $c$. Explicitly,
%there is  a $cs$-bounded cover $\mathcal{U}_k$ of $A_k$ with $s$-multiplicity
%at most $n + 1$. 
%The union of all such covers $\cup_k \mathcal{U}_k$ is obviously  a $cs$-bounded cover   of $Y$.
%From how we choose $r_m$,  the  $s$-multiplicity of such a cover is
%at most $n + 1$. 
%\qed

Next lemma was observed together with Urs Lang. 
We will use the  lemma and the example afterward to show that the Nagata dimension of Carnot groups is equal to the topological one.
\begin{lemma}[{See also \cite[Proposition 2.8]{Lang-Schlichenmaier}}] \label{local}
Let $X$ be a metric space. 
Suppose there exist positive constants $r,c$ and integers $L,n$ such 
that every ball $B(x,r)$, $x \in X$, is doubling with constant $L$ and has 
Nagata dimension at most $n$ with constant $c$. 
Then $X$ has linearly controlled dimension at most~$n$.
\end{lemma}

\begin{proof}
Put $s_0 := r/3$ and choose a set $Z \subseteq X$ that is maximal, 
with respect to inclusion, subject to the condition that distinct points 
in $Z$ are at distance more than $ s_0$ from each other. Then the family of 
balls $B(z,s_0)$, $z \in Z$, covers $X$. By the doubling condition, 
there is an integer $N$ such every ball $B(z,3s_0)$ can be covered by $N$ 
or fewer sets of diameter at most $ s_0$, so $Z \cap B(z,3s_0)$ has cardinality 
at most $N$. It follows that there is a coloring of $Z$ by $N$ colors,
$k \colon Z \to \{1,\dots,N\}$, such that $k(z) \ne k(z')$ whenever
$0 < d(z,z') \le 3s_0$ (see \cite[Lemma~2.4]{Assouad83}). 
Let $C_k$ be the union of all balls $B(z,s_0)$ with $k(z) = k$, 
for $k = 1,\dots,N$. Clearly every $C_k$ has linearly controlled 
dimension at most $n$. Now use finite gluing 
(Corollary \ref{cor:Npart}) 
to show that $X$ has linearly controlled dimension at most~$n$. 
\end{proof}

\begin{example}\label{ex:Liegroup}
Assume $G$ is a Lie group equipped with a left-invariant Riemannian distance. If $n$ is the topological dimension of $G$, then $G$ has linearly controlled dimension $n$. 
Indeed, fixing $r>0$ small enough, all balls $B(x,r)$, for $x\in G$, are isometric to each other  and biLipschitz to an open set of the $n$-dimensional Euclidean space.
In particular, such balls are doubling and have linearly controlled dimension at most $n$, with uniform constants. Hence, Lemma~\ref{local} gives one of the bounds.
Recall that it is a general fact that the topological dimension does not exceed the linearly controlled dimension, see \cite[Theorem 2.2]{Lang-Schlichenmaier}. 
\end{example}

\subsection{Gromov-Hausdorff distance of pointed metric spaces}\label{sec:GH}
For defining limits of unbounded metric spaces, we need to consider pointed metric spaces. A {\em pointed metric space} is a pair $(X,x)$ of a metric space $X=(X,d)$ and a point $x\in X$.  
Recall that, when we are considering a metric space $(X,d)$, we tend to refer to it simply as $X$ whenever it is not necessary to specify the distance. Moreover, we denote by $d$ the distance of any metric space that we are considering.
Likewise, we will denote simply by $X$ a pointed metric space when it is not important to specify the base point. In case we need to denote the base point (and the metric space in discussion is clear) we will use the symbol  $\star$.

% for a pointed metric space
For a set $A$ in a metric space $X$ we denote the $\delta$-neighborhood of $A$ as
\[
 B(A,\delta) := B_X(A,\delta) := \{x \in X \,:\, \dist(x,A)<\delta\}.
\]
Given two metric spaces $X=(X, d_X)$ and $Y=(Y,d_Y)$, we say that $d$ is 
{\em an extension of the distances on } $Y \sqcup   X$ if $d$ is a semidistance on the set $Y \sqcup   X$ (i.e., $d$ might vanish on distinct points) and it coincides with $d_X$ when restricted to $X$ and 
coincides with $d_Y$ when restricted to $Y$. 

We shall use the Gromov-Hausdorff convergence for pointed metric spaces. We will actually want to have a precise distance giving such a topology.
Inspired by the definitions of Gromov and Gabber, see \cite[Section 6]{Gromov-polygrowth}, we define the modified Gromov-Hausdorff measurement:
for all pointed metric spaces  $X=(X, \star_X)$ and $Y=(Y,\star_Y)$, we set
	$$\widetilde{\rm Dist}_{GH} ( \,X, Y \, ) = %\hspace{12cm}$$ $$\hspace{1cm}= 
	\inf\left\{ \epsilon>0 \left|
	\left.\begin{array}{ccc}
	\exists \text{ an extension of   }   \\
	\text{  the distances on }  
	\end{array}\right.
	%\text{ an extension of the distances on } 
	Y \sqcup   X:
	\left.\begin{array}{ccc}
	d(\star_X,\star_Y)\leq \epsilon\\
	B_{  X}(\star_X, 1/\epsilon) \subseteq B_{Y \sqcup   X} ( Y, \epsilon)\\
	B_{Y}(\star_Y, 1/\epsilon) \subseteq B_{Y \sqcup   X} (   X, \epsilon)
	\end{array}\right.\right. \right\}.$$
Probably, the function $\widetilde{\rm Dist}_{GH}$ is not a distance, since it is not clear whether it satisfies the triangle inequality, as Gromov already pointed out.
{However, using the following Lemma~\ref{lma:distghprop} we can easily modify it to be a distance.}
%However, we can easily modify it as follows, since $\widetilde{\rm Dist}_{GH}$ satisfies the following properties:
\begin{lemma}\label{lma:distghprop}
Let $X,Y, Z$ be pointed metric spaces.
\begin{itemize}
\item[(i)] If both $\widetilde{\rm Dist}_{GH}(X,Y), \widetilde{\rm Dist}_{GH}(Y,Z)\leq 1/2$, then
 $$ \widetilde{\rm Dist}_{GH}(X,Z)\leq\widetilde{\rm Dist}_{GH}(X,Y)+ \widetilde{\rm Dist}_{GH}(Y,Z).$$
\item[(ii)] Consequently, the function $\min\{ 1/2 , \widetilde{\rm Dist}_{GH}\}$ satisfies the triangle inequality.
\end{itemize}
\end{lemma} 

\begin{proof}
 The claim follows by considering  distances of the form 
 \[
  d_{X \sqcup   Z}(x,z) := \inf_{y \in Y}\left\{d_{Y \sqcup   X}(x,y) + d_{Z \sqcup  Y}(y,z)\right\}  
 \]
 defined from extensions $d_{Y \sqcup   X}$ and $d_{Z \sqcup  Y}$ of the distances on $X$,$Y$, and $Z$.
\end{proof}

%One can also easily see that $ \widetilde{\rm Dist}_{GH} \leq 1$.{ $\longleftarrow$ This remark seems a bit out of place.}

We define the {\em Gromov-Hausdorff distance} of two pointed metric spaces $X, Y$ as
  $${\rm Dist}_{GH}( \,X, Y \, ):=\min\{ 1/2 , \widetilde{\rm Dist}_{GH}( X, Y  )\}.$$

Other simple properties that hold for $\widetilde{\rm Dist}_{GH}$ and thus for ${\rm Dist}_{GH}$ are the following.
\begin{lemma}\label{lemma_moving}
Let $X$ be a  metric space.
\begin{itemize}
\item[(i)] For all $\lambda_1 >\lambda_2>0$ and all $x\in X$,
 $$  {\rm Dist}_{GH}((\lambda_1 X,x) ,(\lambda_2 X,x) )\leq
\sqrt{\dfrac{\lambda_1}{\lambda_2}-1}.$$
\item[(ii)] For all   $x, x'\in X$,
 $  {\rm Dist}_{GH}(( X,x) ,( X,x') )\leq
d(x,x').$
\item[(iii)] The function
 $ (x, \lambda) \mapsto     ( \lambda X,x) $
 is continuous. In fact, it is H\"older on compact sets of $(0,\infty)\times X$.
\end{itemize}
\end{lemma} 
\begin{proof}
%{\bf TO BE FIXED. Here are two suggestions:}
 To easily  obtain (i) 
 one can isometrically embed $X$ into the Banach space $L^\infty(X)$. % and use the dilations of the ambient space.
The dilations of the ambient Banach space give a straightforward calculation of the Hausdorff distance of 
$\lambda_1X$ and $\lambda_2X$.

%Namely, one considers the extended distance defined between points $y \in \lambda_1X$ and $z \in \lambda_2X$
% defined as
% \[
%  d_{\lambda_1X \sqcup \lambda_2X}(y,z) := (\lambda_1-\lambda_2)d_X(x,y) + \lambda_2d_{X}(y,z),
% \]
% where we have canonically identified $X$, $\lambda_1X$, and $\lambda_1X$ as sets. 
% 
% --- OR ---
% 
% For $\lambda>0$ denote by $y_\lambda$ a point $y \in X$ when considered in the dilated distance $\lambda d$.
% With this notation to obtain (i) it is enough to
% consider the extended distance defined between points $y_{\lambda_1} \in \lambda_1X$ and $z_{\lambda_2} \in \lambda_2X$
% defined as
% \[
%  d_{\lambda_1X \sqcup \lambda_2X}(y_{\lambda_1},z_{\lambda_2}) := (\lambda_1-\lambda_2)d_X(x,y) + \lambda_2d_{X}(y,z).
% \]
%To obtain (i) define the extended distance by setting the distance between any point $y = y_1 = y_2 \in X$ in the two dilated spaces $\lambda_1X$ and $\lambda_2X$
% to be $d_{\lambda_1X \sqcup \lambda_2X}(y_1,y_2) := (\lambda_1-\lambda_2)d_X(x,y)$ and by defining it as
% $d_{\lambda_1X \sqcup \lambda_2X}(y_1,z_2) := d_{\lambda_1X \sqcup \lambda_2X}(y_1,y_2) + d_{\lambda_2X}(y_2,z_2)$ for dilatations of distinct points $y,z \in X$.
 To see (ii) use the original distance as the extension. 
 
 The claim (iii) follows from (i) and (ii).
\end{proof}

 To define tangent metric spaces, we dilate a metric space and   consider the accumulation points of such sequences of dilated spaces.
  Recall that whenever the distance and the base point are clear we just write 
$X$ to denote the %$(X,d_X,x)$
%Given a 
pointed metric space $(X,d,\star)$, or we write $d=d_X$ to emphasize that $d$ is the distance on $X$. 
Moreover, we denote by $\lambda X$ the pointed metric space obtained by dilating the distance by $\lambda>0$. Namely, we set $\lambda X:=(X,\lambda d_X,\star).$ 
  
   \begin{definition}[Tangent metric spaces]\label{tangent spaces} Let  $(X,d_X)$ be a metric space and  $x$ a point of it.
Then a pointed metric space
$(Y,   d_Y,y)$ is said to be %{\color{\bf a}}  
 {\em a tangent} of $(X,d_X)$ at $x$
if 
$(X,\lambda d_X,x)$
accumulate  to 
$(Y,   d_Y,y)$ in the pointed Gromov-Hausdorff topology, as $\lambda \to \infty$.
Namely, there exists a diverging sequence $\lambda_j$ such that 
 $${\rm Dist}_{GH}(( X ,\lambda_j d_X ,x) ,( Y,d_Y,y) ) \to 0 , \, \text{ as } j\to \infty.$$
 We denote by %$\mathcal T_xX$, or 
 ${\rm Tan}(X,x)$ the collection of all tangents of $(X,d_X)$ at $x$.
 \end{definition}
 
It is easy to come up with examples of metric spaces with more than one tangent at a given point.
 Here is an easy criterion (which follows from Lemma~\ref{lemma_moving}(i)) to conclude that the tangent at a point is unique.
 \begin{lemma}
Let $X$ and $Y$ be two pointed metric spaces.
\begin{itemize}
\item[(i)] If $n X \to Y$, as $n\to \infty, n\in \N$, then 
$\lambda X \to Y$, as $\lambda\to \infty, \lambda\in \R$.
\item[(ii)] More generally, let $a_n$ be a diverging sequence such that 
$a_{n+1}/a_n \to 1$, as $n\to \infty$.
If $a_n X \to Y$, as $n\to \infty$, then 
$\lambda X \to Y$, as $\lambda\to \infty, \lambda\in \R$. 
\end{itemize}
\end{lemma}

\subsection{Some remarks on the dimensions of limits}

In this section we prove that the dimensions of the tangents bound from below the dimension of a space. 
Regarding the Assouad dimension, this fact is a quantified version of the well-known result that 
%Recalling that % In other papers it is referred as the fact that 
limits of doubling metric spaces, with uniform constants, are doubling.  
\begin{lemma}\label{lem_ass_limit}
 Suppose that a sequence of pointed metric spaces $X_j$ converges to a pointed metric space $X_\infty$. 
Let  $\beta\geq0$, $C>1$, and $\bar R>0$.
If, %that there exist two constants $C=C_\beta>1$ and $\bar R_\beta>0$ such 
 for all $j$, 
 $\dim_A(X_j,C,  \bar R)\leq 
\beta$, then
 $\dim_A(X_\infty,C,  \bar R)\leq 
\beta$.
%
%
%that, for all $0<r<R<\bar R$ and for all $j \in \N$, any ball of radius $R$ in $X_j$ can be covered with
%less than  $C\left( {R}/{r}\right)^\beta$ balls of radius $r$ in $X_j$.
%Then for all $0<r<R<\bar R$, any ball of radius $R$ in $X_\infty$ can be covered with 
%less than $C\left( {R}/{r}\right)^\beta$ balls of radius $r$ in $X_\infty$.
 \end{lemma}
\begin{proof}
Take $0 < r < R < \bar R$ and $x \in X_\infty$. For all $j \in \N$, let $d_j$ be a distance on $X_j \sqcup X_\infty$
extending the distances on $X_j$ and $X_\infty$ such that $d(\star_{X_j},\star_{X_\infty}) \le \epsilon_j$,
	$B_{{X_j}}(\star_{X_j}, 1/\epsilon_j) \subseteq B_{{X_\infty} \sqcup   {X_j}} ( {X_\infty}, \epsilon_j)$, and 
	$B_{{X_\infty}}(\star_{X_\infty}, 1/\epsilon_j) \subseteq B_{{X_\infty} \sqcup   {X_j}} (   {X_j}, \epsilon_j)$
 for some sequence $\epsilon_j$  %$(\epsilon_j)_{j=1}^\infty$
 going to $0$.
%  such that $\epsilon_j \to 0$ as $j \to \infty$.

 Now for $j \in \N$ large enough we have $B_{X_\infty}(x,R) \subseteq B_{{X_\infty} \sqcup   {X_j}} ( {X_j}, \epsilon_j)$.
 Take $x_j \in X_j$ with $d_j(x,x_j)< \epsilon_j$. For $j$ large enough we have $R + 2\epsilon_j < \bar R$ and 
 we need
 at most
 $C\left(\frac{R+2\epsilon_j}{r-2\epsilon_j}\right)^\beta$ 
  points $x_{j,i}\in B_{X_j}(x, R+r) \subset B_{{X_j}}(\star_{X_j}, 1/\epsilon_j)$ such that
 %by assumption
  the ball $B_{X_j}(x_j,R+2\epsilon_j)$ is covered by the balls $B_{X_j}(x_{j,i},r-2\epsilon_j)$.
 Select points $\tilde x_{j,i} \in X_\infty$ with $d_j(x_{j,i}, \tilde x_{j,i}) < \epsilon_j$.
We claim that $\{B_{X_\infty}(\tilde x_{j,i},r)\}_i$ is a cover for the ball $B_{X_\infty}(x,R)$.

% Let us check that this is indeed the case. T
 Indeed, take any $y \in B_{X_\infty}(x,R)$ and $y_j \in X_j$ with $d_j(y,y_j)<\epsilon_j$.
 Now
 \[
  d_j(y_j,x_j) < d_j(y_j,y) + d_j(y,x) + d_j(x,x_j) \le \epsilon_j + R + \epsilon_j = R + 2\epsilon_j,
 \]
 and so $y_j \in B_{X_j}(x_{j,i},r-2\epsilon_j)$ for some $i$. Because 
 \[
  d_j(y,\tilde x_{j,i}) < d_j(y,y_j) + d_j(y_j,x_{j,i}) + d_j(x_{j,i},\tilde x_{j,i}) \le \epsilon_j + r-2\epsilon_j + \epsilon_j = r,
 \]
 we have $y \in B_{X_\infty}(\tilde x_{j,i},r)$.

 Since $\epsilon_j \to 0$ as $j \to \infty$, with large enough $j$ we obtain a cover of $B_{X_\infty}(x,R)$ with no more than
 $C\left( {R}/{r}\right)^\beta$ balls of radius $r$.
\end{proof}

The following consequence can be also found in \cite[Proposition 6.1.5]{Mackay_Tyson}.
\begin{corollary}\label{cor:Assouadlower}
 The Assouad dimension of any tangent space of a metric space $X$ does not exceed the Assouad dimension of $X$.
\end{corollary}

Regarding the  Nagata dimension, the similar  bound is slightly less trivial and is based on
 Lemma~\ref{union}.

\begin{proposition}
 Suppose that a sequence of pointed metric spaces $X_j$ coverges to a pointed metric space $X_\infty$. %If all spaces $X_j$ have Nagata dimension equal to $n$ for the same constant $c$, then the Nagata dimension of $X_\infty$ is at most $n$.   
 Let $0 < s_1 < \infty$ and $0<c<1$. Assume that
 \[
  \dim_N(X_j,c,s) \le n, \qquad \text{for all }0 < s \le s_1 \text{ and all }j = 1,2,\dots.
 \]
 Then, for any $0 < c' < c^2/5$, we have
 \[
  \dim_N\left(X_\infty,c',s\right) \le n, \qquad \text{for all } 0 < s < \left(1+\frac23c\right)s_1.
 \]
 %Namely, with the notation of Definition \ref{def:Nag_scale},
 %  \begin{equation*}%\label{lma:subset}
 % \dim_N(X_j,c,s) \le n \implies    \dim_N(X_\infty,c,s) \le n.
 %\end{equation*}
\end{proposition}
\begin{proof}
 Our aim is to show that for all $0 < c'' < c$ we have
 \begin{equation}\label{eq:nagaconvaim}
  \dim_N(B_{X_\infty}(\star_{X_\infty},k),c'',s) \le n, \qquad \text{for all }0 < s < s_1 \text{ and all }k = 1,2,\dots.
 \end{equation}
 Once we have obtained this, the claim will follow from Lemma~\ref{union}.

 Fix an integer $k > 0$ and $0 < c'' < c$. 
 For all $j \in \N$, let $d_j$ be a distance on $X_j \sqcup X_\infty$
 extending the distances on $X_j$ and $X_\infty$ such that $d(\star_{X_j},\star_{X_\infty}) \le \epsilon_j$,
 $B_{{X_j}}(\star_{X_j}, 1/\epsilon_j) \subseteq B_{{X_\infty} \sqcup   {X_j}} ( {X_\infty}, \epsilon_j)$, and 
 $B_{{X_\infty}}(\star_{X_\infty}, 1/\epsilon_j) \subseteq B_{{X_\infty} \sqcup   {X_j}} (   {X_j}, \epsilon_j)$
 for some sequence $\epsilon_j$ going to $0$. %such that $\epsilon_j \to 0$, as $j \to \infty$.

 Take $0 < s < s_1$ and $j \in \N$ so that $\epsilon_j < \frac1k$ and
 \begin{equation}\label{eq:ccprimeprime}
  cs-2(1+c)\epsilon_j  \ge c''s.
 \end{equation}
 Let $\bigcup_{i=0}^n \mathcal{B}_i$ be an $(s-2\epsilon_j)$-bounded cover
 of $X_j$ where each  $\mathcal{B}_i$ is $c(s-2\epsilon_j)$-separated. Define the new collections of sets $\mathcal{B}_i'$
 as
 \[
  \mathcal{B}_i' := \left\{  B_{X_\infty}(\star_{X_\infty},k)
  \cap B_{{X_\infty} \sqcup   {X_j}} (   {B}, \epsilon_j)
\right\}_{B \in \mathcal{B}_i}.
 \]
 Now, for each $x \in B_{X_\infty}(\star_{X_\infty},k)$, there exists $x_j \in X$ with $d_j(x,x_j)<\epsilon_j$.
 Because $x_j \in B$ for some $i = 0,\dots,n$ and $B \in \mathcal{B}_i$, we have $x \in B'$ for some 
 $i = 0,\dots,n$ and $B' \in \mathcal{B}_i'$. Therefore $\bigcup_{i=0}^n\mathcal{B}_i'$ is an $s$-bounded
 cover of $B_{X_\infty}(\star_{X_\infty},k)$ where each   $\mathcal{B}_i'$ is
 $c''s$-separated, by \eqref{eq:ccprimeprime}.
 %Letting $j \to \infty$ proves \eqref{eq:nagaconvaim} and the claim of the proposition.
\end{proof}

\begin{corollary}\label{cor:Nagatalower}
 The Nagata dimension of any tangent space of a metric space $X$ does not exceed the Nagata dimension of $X$.
\end{corollary}

%\begin{note}
%Enrico would put here Prop 4.2 (1-dim'l tang)
%\end{note}

As shown by the next example, in general one cannot hope to deduce
an upper bound for the   dimension of the space by simply looking at the   dimensions of the tangents.

\begin{example}\label{ex:nonuniformtangent}
 Define a set $X \subset \R$ as
 \[
  X = \{0\} \cup \bigcup_{i = 1}^\infty \bigcup_{k = 1}^i\{2^{-i^2}+k2^{-i^3}\}.
 \]
 Let the distance be induced by the Euclidean distance on $\R$.
 Then $X$ is compact, $\dim_A X =\dim_NX = 1$ and 
 ${\rm Tan}(X,x) \subset \{\{0,t\} \subset\R\,:\, t \in [0,\infty)\}$ for all $x \in X$.
 Hence, $\dim_A Y =\dim_NY = 0$, for all $Y  \in {\rm Tan}(X,x)$.
 In particular,
 \begin{equation}\label{eq:nonequaltangents}
   \max\{\dim Y\,:\,Y \in {\rm Tan}(X,x), x \in X\} < \dim X,
 \end{equation}
 where $\dim$ is either the Nagata or the Assouad dimension.
\end{example}

Notice that the space in Example \ref{ex:nonuniformtangent} has multiple tangents at $0$, the convergence
to the tangents is not uniform, and that in the definition of a tangent space we keep the base point fixed. 
Hence, if we want to obtain the 
Nagata dimension of the space as the maximum of the dimensions of the tangents, we need to
either change the notion of tangents (e.g., consider weak tangents) or   impose more restrictions on the convergence to the tangents.
Regarding this last option, we
%Most importantly we
 will study the
case of uniformly close tangent and, in the subsequent Section \ref{sec:equiSR}, we will apply our results to subRiemannian manifolds.
 
Observe also that in Example \ref{ex:nonuniformtangent}, if we take a larger class of tangents where we allow the change
of the the base point, we obtain equality in \eqref{eq:nonequaltangents}.
Simply let $r_i = i2^{-i^3}$ and $x_i = 2^{-i^2}$. This observation is true    in compact doubling metric spaces when
$\dim_NX = 1$.
This is the content of the following proposition, where we consider {\em weak tangents}, i.e.,  limits of the forms $(\frac{1}{r_i} X, x_i)$.

\begin{proposition}\label{prop:weaktangentnagata}
 Let $X$ be a doubling metric space with $\dim_{LC}X \ge 1$. Then there exists a sequence of points $x_i \in X$ 
 and a sequence $r_i \searrow 0$ such that the sequence $(\frac{1}{r_i} X, x_i)$ converges to a space with linearly controlled dimension at least $1$.
\end{proposition}
\begin{proof}
For each $x \in X$ and $\delta > 0$, define the iterated $\delta$-neighborhoods of $x$ by setting $N_X^0(x,\delta) := \{x\}$ and
\[
 N_X^n(x,\delta) := B(N_X^{n-1}(x,\delta), \delta), \qquad \text{for }n \ge 1.
% \{y \in X \,:\, \dist(y,N_X^{n-1}(x,\delta))<\delta\}, \qquad \text{for }n \ge 1.
\]
Define also $N_X^\infty(x,\delta) := \bigcup_{i=1}^\infty N_X^n(x,\delta)$.
Take $i \in \N$. Since $\dim_{LC}X \ge 1$, we claim that there exist a point $x_i \in X$ and a radius $r_i < \frac1i$ such that
$\diam(N_X^\infty(x_i,r_i/i)) > r_i$. Indeed, suppose that this is not true. Then for every $r < \frac1i$
the collection $\{N_X^\infty(x,r/i)\,:\, x \in X\}$ would be an $r$-bounded $\frac{r}{2i}$-separated cover of $X$.
This would mean that $\dim_{LC}X = 0$.

By Gromov Theorem, see \cite[Theorem 8.1.10]{Burago:book},
because of the doubling assumption the sequence $(r_i^{-1} X, x_i)$ has a subsequence converging to a pointed metric space $(Z,z)$.
Now, for any $0 < \epsilon < \frac12$, there is $i > \frac1\epsilon$ such that 
\[
{\rm Dist}_{GH}\left((r_i^{-1} X, x_i),(Z,z)\right) < \epsilon.
\]
Consequently, there exists a distance $d$ extending the distances on $r_i^{-1}X$ and $Z$ such that
$d(x_i,z)\leq \epsilon$ and
$B_{r_i^{-1}X}(x_i, 1/\epsilon) \subseteq B_{Z \sqcup   r_i^{-1}X} (Z, \epsilon)$.
Now, for each 
\[
 x' \in N_{r_i^{-1}X}^\infty(x_i,1/i) \cap B_{r_i^{-1}X}(x_i,2) = N_X^\infty(x_i,r_i/i)\cap B_X(x_i,2r_i) ,
\]
there exists $z' \in Z$ with $d(x',z')<\epsilon$. Thus $\diam(N_Z^\infty(z,3\epsilon)) \ge 1 - 2\epsilon$.
Letting $\epsilon \searrow 0$ shows that $\dim_{LC}Z \ge 1$.
\end{proof}

We do not know if Proposition \ref{prop:weaktangentnagata} is true  with higher lower bound on the dimension.
We  were unable to  find a way to generalize our argument to the higher dimensional case.
 
\begin{question}\label{q:movingblowup}
 Let $X$ be a metric space. Does there exist a sequence of points $x_i \in X$ and a sequence $r_i \searrow 0$
 such that the weak tangent of $X$ along the sequences $x_i$, $r_i$ has exactly the same linearly controlled dimension as $X$?
\end{question}

\section{Dimension of uniformly close tangents}\label{sec:uniform}
As we saw in Example \ref{ex:nonuniformtangent}, both the Nagata dimension and the Assouad dimension of a space can in general be strictly larger than the supremum of the dimensions of its tangents.
We wonder when such a supremum equals the dimension of the space.
A key assumption that we will make, which will still not be enough,  is that the convergence to the tangents is uniform, as we now explain.
Recall that ${\rm Dist}_{GH}$ is the distance defined in Section \ref{sec:GH}.

%Let us now turn to the main theme of this section.
 
 Let $X$ be a metric space. We say that the metric space $X$ has {\em uniformly close tangents in} $K\subset X$ if we have
$$\lim_{\lambda\to \infty} {\rm Dist}_{GH}( (\lambda X,x), {\rm Tan}(X,x)) =0 $$ uniformly in $x \in K$.
If $X$ has uniformly close tangents in $X$, we simply say that $X$ has uniformly close tangents.
%{We want to see whether having uniformly close tangents is sufficient
%to imply that we get a bound on the Nagata  (resp.  Assouad) dimension of the space from the Nagata  (resp.  Assouad) dimensions of its tangents.
%Before investigating this, we list some basic properties of uniformly close tangents. (Above we already noted that unif. tan. are not enough.)}
Before investigating what other assumptions we need to make, we list some basic properties of uniformly close tangents.
 
 \begin{lemma}\label{lma:continuoustangents} Assume that a metric space $X$ has uniformly close tangents.
  If, for all $x\in X$, there exists only one element $T_x\in  {\rm Tan}(X,x)$,
then the function
$x\mapsto T_x$ is continuous. % ({\bf check!}) 
 \end{lemma}
We omit the easy proof of the above fact, which is proved exactly as one proves that the uniform limit of continuous maps is continuous, via Lemma~\ref{lemma_moving}(ii).
 In general, without the uniqueness assumption on the tangents, one cannot even conclude that the
 ``graph''  $\{(x,T) :x\in X, T\in {\rm Tan}(X,x)\}$ is closed.
Look for example at Example \ref{ex:nagata}.  

%Observe that the property of being isometrically homogeneous is stable under Gromov-Hausdorff convergence.
From \cite{LeDonne6} we know that if a complete metric space admits a doubling measure $\mu$ and
if the space has unique tangents, then at $\mu$-almost every point the unique tangent is an isometrically homogeneous
space admitting dilations.
We have the following version of this result.

%Therefore, combining \cite[Theorem 1.4]{LeDonne6} with Lemma~\ref{lma:continuoustangents} and
%the fact that complete metrically doubling spaces admit 
%doubling measures (see for instance \cite{VolbergKonyagin}, \cite{LuukkainenSaksman}, or \cite{KaenmakiRajalaSuomala2012}),
% we have the following result.

\begin{proposition}
  Let $X$ be a complete doubling metric space with uniformly close tangents such that,
 for all $x \in X$, there exists only one element $T_x\in  {\rm Tan}(X,x)$.
 Then, for all $x \in X$, the tangent $T_x$ is an isometrically homogeneous space admitting dilations.
\end{proposition}

\begin{proof}
 Since $X$ is a complete doubling metric space there exists a doubling measure $\mu$ on $X$ 
 (see for instance \cite{VolbergKonyagin}, \cite{LuukkainenSaksman}, or \cite{KaenmakiRajalaSuomala2012}).
 By \cite[Theorem 1.4]{LeDonne6} we already know that at $\mu$-almost every $x \in X$
 the tangent $T_x$ is an isometrically homogeneous space.
 By Lemma~\ref{lma:continuoustangents} we know that all the tangents $T_x$ are Gromov-Hausdorff limits of
 isometrically homogeneous spaces. Since being isometrically homogeneous is stable under Gromov-Hausdorff convergence,
 all the tangents are isometrically homogeneous. The fact that the tangents admit dilations follows
 directly from the fact that ${\rm Tan}(X,x)$ are singletons.
\end{proof}

\subsection{Nagata dimension and uniformly close tangents}\label{Sec:Nagata_uniform_tangents}

It turns out that having uniformly close tangents is not enough to tie the Nagata dimension of the space to the Nagata dimension
of its tangents. This is shown by the next example.

\begin{example}\label{ex:nagata}
 Let us define a metric space $X \subset \R^2$. Define for all $n \in \N$ a basic construction piece $S_n$, in polar coordinates as
 \[
  S_n = \left\{(1,2\pi k/n) \,:\,k= 1, \dots, n\right\}.
 \]
 The set $S_n$ consists of $n$ equally distributed points on the unit circle in $\R^2$.
 %The set $S_n$ looks the same at all the points $x \in S_n$.

 Using the sets $S_n$ define for all $n \in \N$ the set
 \[
  E_n = \left\{\sum_{i=n}^\infty 2^{-i^2}A_i \,:\, A_i = S_n \text{ if } i \text{ odd, and }A_i = \{(0,0),(1,0)\}\text{ if }i\text{ even }\right\}
 \]
 and from these finally the space
 \[
  X = \cl\left({\bigcup_{n=1}^\infty E_n + \left(2^{-n^2},0\right)}\right).
 \]
 The construction of $X$ is illustrated in Figure \ref{fig:example3}.

\begin{figure}\label{fig:example3}
  \centering
  %\psfrag{e1}{$E_1$}
  %\psfrag{e2}{$E_2$}
  %\psfrag{e3}{$E_3$}
  %\psfrag{e16}{$E_{16}$}
  %\psfrag{0}{$0$}
  %\includegraphics[scale=.7]{figure}
  \includegraphics{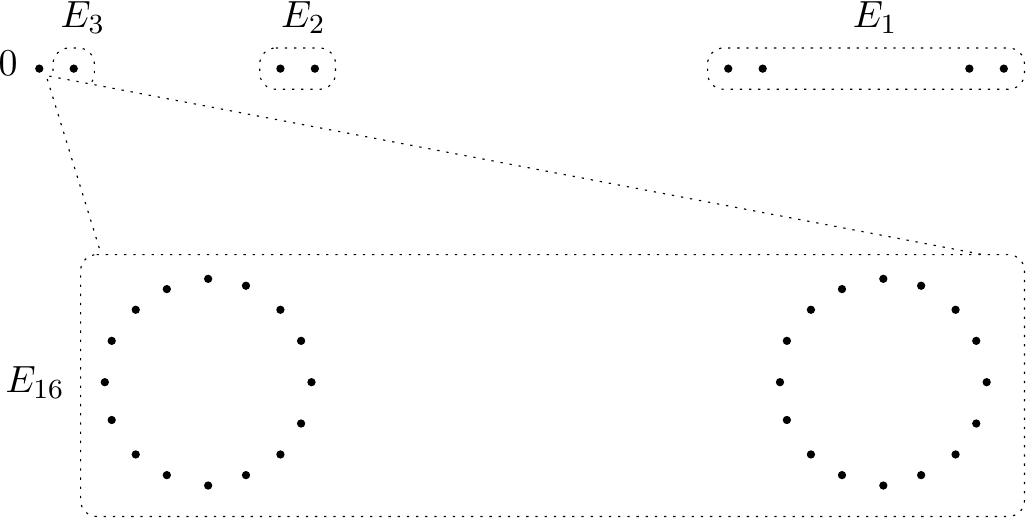}
  \caption{%{add more stuff to the illustration.}
    Illustration of the space $X$ in Example \ref{ex:nagata}. It consists of an infinite number of construction pieces $E_i$,
    each of which is constructed by alternating in taking $i$ points on a sphere or just two points. The magnified illustration of the construction piece $E_{16}$ shows this idea.
 %in the compact space with uniformly close tangents but no uniform constant for the Nagata dimension estimates.
}
\end{figure}
	
 As tangents for every $x \in \cl{(E_n + 2^{-n^2})}$ we have
 \[
  {\rm Tan}(X,x) = \{tS_n\,:\, t\in (0,\infty)\} \cup \left\{\{0,t\} \,:\, t \in [0,\infty)\right\},
 \]
 and for the origin
\[
  {\rm Tan}(X,(0,0)) = \left\{\{0,t\} \,:\, t \in [0,\infty)\right\},
\]
 We claim that the space $X$ (with the distance induced by the Euclidean distance on $\R^2$)
 has uniformly close tangents. To see this, take  $\lambda>0$ and $x \in X$. Let $i \in \N$
 be such that $2^{i^2-i} < \lambda \le 2^{(i+1)^2-i-1}$. Then $\diam(\lambda 2^{-j^2}A_j) \le 2^{-i}$ for all $j \ge i+1$,
 so down to scale $2^{-i}$ we can consider the sets $2^{-j^2}A_j$, $j \ge i+1$, to be just points in the $\lambda$-dilated distance.
 On the other hand, since any two distinct points in $A_{i-1}$ have distance at least $\frac1{i-1}$ between them
 and $\lambda \frac{1}{i-1}2^{-(i-1)^2} \ge \frac{1}{i-1}2^{i-1}$, there is at most one point of $2^{-(i-1)^2}A_{i-1}$ inside
 a ball of radius $\frac1{1-i}2^{i-2}$ in the $\lambda$ dilated distance. In particular,
 \[
  {\rm Dist}_{GH}((\lambda X,x), {\rm Tan}(X,x)) \le (1-i)2^{2-i}
 \]
for all $x \in X$. Hence $X$ has uniformly close tangents.

The Nagata dimension of any of the tangents of $X$ is zero. However,  the  Nagata dimension and the linearly controlled dimension of the space are one.
\end{example}

In Example \ref{ex:nagata} the tangents of the space have Nagata dimension zero with smaller and smaller
constant $c$ as we move the base point towards $(0,0)$. On the other hand, the set of tangents is
more than a singleton at every point in the space.
In the following Theorem~\ref{thm:uniformtangents} we show that if we rule out one of the above mentioned properties,
namely if we require either  uniformity of the constant $c$ or the uniqueness of tangents, the 
linearly controlled dimension of the space
is bounded above by the supremum of the dimensions of its tangents.

\begin{theorem}\label{thm:uniformtangents}
Let $X$ be a metric space.
% Suppose that  with $\dim_{LC}(X) < \infty$.
 Let $K\subset X$ be a subset on which the convergence to tangents is uniform.
Assume that $\dim_{LC} K < \infty$ and that one of the  following two situation  holds:
\begin{enumerate}
 \item[(i)]  $\dim_{LC}(B_{Y}(\star,1)) \le n$, for all $x \in K$ and $Y \in {\rm Tan}(X,x)$, with uniform constants $c$ and $s$, 
 or
%       then $\dim_{LC}(X) \le n$.
 \item[(ii)] $K$ is compact, % set %$K \subset X$ so that 
 for every $x \in K$, the tangent is unique, i.e., the set ${\rm Tan}(X,x)$ is a singleton denoted by $T_xX$,
       and $\dim_{LC}(B_{T_xX}(\star,1)) \le n$. %, then $\dim_{LC}(K) \le n$.
\end{enumerate}
 Then $\dim_{LC} K \le n$.
\end{theorem}

 For proving the above theorem, we need the following two lemmas.

\begin{lemma}\label{lma:dimfortangents}
 Let $(X,x)$ and $(Y,y)$ be two pointed   metric spaces.
Suppose that there exist some $\epsilon, r,c,s,n$ such that 
\[
 {\rm Dist}_{GH}(( X,x) ,( Y,y) )<  \epsilon \qquad \text{and} \qquad \dim_N(B_Y(y,r),c,s) \le n.
\]
 Then, if
 $r' \leq \min\{ 1/\eps, r-2\eps\}$, $c':= \frac{cs - 2\epsilon}{s + 2\epsilon}$,
 and $s':=s+2\epsilon,$ we have
 \[
  \dim_N\left(B_X(x,r'),c' ,s'\right) \le n.
 \]
\end{lemma}

\begin{proof}
Since $( X,x) $ and $( Y,y) $ have distance $< \epsilon$, we can see them as a subset of a metric space $Z$ such that $	d(x,y)\leq \epsilon$,
 and $B_{X}(x, 1/\epsilon) \subseteq B_{Z} (   Y, \epsilon)$.

We claim that 
\begin{equation}\label{cheneso}
B_{X}(x, r') \subseteq B_{Z} (   B_Y(y,r), \epsilon).
\end{equation}
Indeed, we have 
$B_{X}(x, r')   \subseteq B_{X}(x, 1/\epsilon) \subseteq B_{Z} (  Y, \epsilon)$ and hence, for all 
$x'\in B_{X}(x, r') $, there is $y'\in X$ such that
 $	d(x',y')\leq \epsilon$.
 Such an $y'$ is such that 
 $	d(y',y)\leq r' +2\epsilon$.
 Hence,
 $B_{X}(x, r')   \subseteq  B_{Z} (   B_Y(y,r' +2\epsilon ), \epsilon)  \subseteq  B_{Z} (   B_Y(y,r ), \epsilon). $ 

Let $\mathcal U$ be an $s$-bounded cover of $B_Y(y,r)$ such that 
 $\mathcal U = \mathcal U_0 \sqcup\ldots\sqcup \mathcal U_n $ with each $\mathcal U_j$ 
$cs$-separated.
Set  $\mathcal V_j:=\{ B_Z(U,\eps) \cap X :  U\in  \mathcal U_j\}.$
Clearly, $\mathcal V_j$ are $(s+2\eps)$-bounded and $(cs -2\eps) $-separated.
Since \eqref{cheneso}, 
the family $ \mathcal V_0 \cup\ldots\cup \mathcal V_n $ gives a cover of  $B_{X}(x, r') $.
\end{proof}

%\begin{proof}
% Take subsets $\{U_i^0\}_{i \in I_0}, \dots, \{U_i^n\}_{i \in I_n}$ of $X$ with the properties
% \[
%  \diam(U_i^j) \le s \qquad \text{for all }j = 0, \dots, n \text{ and }i \in I_j,
% \]
% \[
% \dist(U_i^j, U_k^j) \ge cs\qquad   \text{for all }j = 0, \dots, n \text{ and }i,k \in I_j \text{ with }i \ne k
% \]
% and
% \[
%  X = \bigcup_{j = 0}^n\bigcup_{i \in I_j}U_i^j.
% \]
% Let $i_1 \colon X \to Z$ and $i_2 \colon Y \to Z$ be isometric embeddings so that $d_H(i_1(X),i_2(Y)) < 2\epsilon$.
% Define for all $j = 0, \dots, n$ and $i \in I_j$
% \[
%  V_i^j := i_2^{-1}(i_1(U_i^j)(2\epsilon)).
% \]
% Then
% \[
%  \diam(V_i^j) \le s + 4\epsilon \qquad \text{for all }j = 0, \dots, n \text{ and }i \in I_j,
% \]
% \[
% \dist(V_i^j, V_k^j) \ge cs - 4\epsilon\qquad   \text{for all }j = 0, \dots, n \text{ and }i,k \in I_j \text{ with }i \ne k
% \]
% and
% \[
%  Y = \bigcup_{j = 0}^n\bigcup_{i \in I_j}V_i^j.
% \]
%\end{proof}

\begin{lemma}\label{lma:nagatafinalconclusion}
 For every choice of constants $\overline{n}, n \in \N$, $ \overline{s}, s_0\in (0,\infty]$, $a, b> 0$ and $c, \overline{c} \in (0,1)$
 there exists $\epsilon \in (0,\frac18)$ such that the following holds.

 Let $K \subset X$. Suppose that 
 \begin{equation}\label{eq:Xdimbound}
   \dim_N(K,\overline{c},s) \le \overline{n}, \qquad \text{for all } 0 < s \le \overline{s},
 \end{equation} 
 and that there exists $r_\epsilon > 0$ such that, for all $x\in K$, for all $r\in (0, r_\eps)$, and for all $s\in (0, s_0)$,
 \begin{equation}\label{eq:dimforballs}
   \dim_N\left(B(x,r/2),\frac{cas-b\epsilon}{as + b\epsilon},r(as + b\epsilon)\right) \le n.
 \end{equation}
 Then 
 \[
  \dim_{LC} K \le n.
 \]
\end{lemma}
\begin{proof}
Take $0 < r < \min\{r_\eps,\overline{s}\}$. From  \eqref{eq:Xdimbound}, we have  an $\frac{r}{2}$-bounded cover
$\mathcal U = \mathcal U_0 \sqcup\ldots\sqcup \mathcal U_{\bar n} $ of $K$  where each $\mathcal U_j$ is 
$\frac {\overline{c}r}{2}$-separated. 
%
% $\{U_i^0\}_{i \in I_0}, \dots, \{U_i^{\overline{n}}\}_{i \in I_{\overline{n}}}$   of $K$ 
% given by \eqref{eq:Xdimbound} with the properties
% \[
%  \diam(U_i^j) \le r/2 \qquad \text{for all }j = 0, \dots, \overline{n} \text{ and }i \in I_j,
% \]
% \[
% \dist(U_i^j, U_k^j) \ge \overline{c}r/2\qquad   \text{for all }j = 0, \dots, \overline{n} \text{ and }i,k \in I_j \text{ with }i \ne k.
% \]
% and
% \[
%  K = \bigcup_{j = 0}^{\overline{n}}\bigcup_{i \in I_j}U_i^j.
% \]
 By \eqref{eq:dimforballs} we have,  for all  $ U \in \mathcal U $
 \[
   \dim_N\left(U,\frac{cas-b\epsilon}{as + b\epsilon},r(as + b\epsilon)\right) \le n, \qquad \text{ for all } 0 < s < s_0.
 \]
 %, $j = 0, \dots, \overline{n}$ and $i \in I_j$. 
For all $j = 0, \dots, \overline{n}$,  setting
 $U^j := \bigcup_{U \in \mathcal U_j}U$,
  %$U^j = \bigcup_{i \in I_j}U_i^j$,  
 from Lemma~\ref{lma:dimforseparated} we get
 \[
  \dim_N\left(U^j,\min \left\{\frac{cas-b\epsilon}{as + b\epsilon},\dfrac{\overline{c}}{2 (as + b\epsilon ) }\right\},r(as + b\epsilon)\right) \le n,
  \quad \text{  for all  }0 < s < s_0. % \text{  and } j = 0, \dots, \overline{n}.
 \]
 %for all $0 < s < s_0$ and $j = 0, \dots, \overline{n}$.
 
 By decreasing $s$ and $\eps$ to be small enough, we may assume that   $2 (as + b\epsilon) < 1$.
Set    $ {s}_\eps : = \frac{b (2+c)}{ac}\epsilon$. Then one can easily check that, if $s>s_\eps$,
 % and notice that with this choice
 \[
  \frac{ac {s} -b\epsilon}{ a{s}  + b\epsilon}    \ge \frac{c}{2}.
 \]
 Indeed, one just needs to verify that, by substituting the value of $\eps $ in terms of $s_\eps$, the above inequality is equivalent to  $s>s_\eps$.
 Therefore, setting  $
  c_1 = \min\left\{\frac{c}{2},\overline{c}\right\}
 $
 we have that, for all $s$ between $s_\eps$ and $s_0$ and all  $j = 0, \dots, \overline{n}$, we have
 \[
  \dim_N\left(U^j, c_1,r({as + b\epsilon})\right) \le n.
 \]
  %for all $\overline{s}_0 < s < s_0$ and, where
 From Corollary \ref{cor:Npart} it then follows that
 \[
  \dim_N\left(K, \frac{c_1^{\overline{n}}}{5^{\overline{n}-1}},r({as + b\epsilon})\right) \le n,
 \]
 for all
 \begin{equation}\label{eq:interval}
 {\left(2 + \frac3{c_1}\right)^{\overline{n} - 1}  r(a{s}_\eps + b\epsilon) \le  r(as + b\epsilon) \le r(as_0+b\epsilon).}
 \end{equation}
 In order to have $\dim_{LC} X \le n$ it now suffices to select the constant $\epsilon$ so that
 there exists some $s$ satisfying \eqref{eq:interval}. This is the case if we take
 \[
   {\epsilon \le as_0\left(\left(2 + \frac3{c_1}\right)^{\overline{n}-1 }   \left(\frac{b(2+c)}{c} + b\right) - b \right)^{-1}.}
%  \epsilon \le \frac{cs_0}{8\left(2 + \frac3{c_1}\right)^{\overline{n} - 1}}.
 \] 
\end{proof}

\noindent
{\bf Proof of Theorem~\ref{thm:uniformtangents}.}
 We shall prove the claims by checking that the assumptions of Lemma~\ref{lma:nagatafinalconclusion} hold.
 First of all, the assumption $\dim_{LC} X < \infty$ implies that
 there exist $\overline{n},   \overline{s},  \overline{c}$ such that
  \eqref{eq:Xdimbound} is satisfied in both cases of the theorem.
 
 Since by assumption   the tangents are uniformly close  in $K$, for any $\eps>0$ there exists $\lambda_\eps$ such that, for all $\lambda\geq\lambda_\eps$ and all $x\in K$, 
 \begin{equation}\label{eq:tangentivicini}
 {\rm Dist}_{GH}(( \lambda X,x) ,{\rm Tan}( X,x) )<\eps.
 \end{equation}
 
We now consider separately the two situations assumed in the hypothesis of the theorem.
In the first case we will show that  \eqref{eq:dimforballs} holds. So let $\epsilon \in (0,\frac18)$ be fixed.
In this case we are assuming that we already have $c$ and $s_0$ such that, for all $x\in K$ and for all $Y\in {\rm Tan}(X,x)$, 
 \begin{equation} \label{eq:unifdimfortan}
  \dim_N(B_Y(\star, 1),c,s) \le n, \qquad \text{for all } 0 < s < s_0. 
   \end{equation}
By Lemma~\ref{lma:dimfortangents}, from \eqref{eq:tangentivicini} and \eqref{eq:unifdimfortan},
since $1/2 < \min\{1/\eps , 1-2\eps\}$, we have 
  \begin{equation*}%\label{eq:dimforballs}
   \dim_N\left(B_{\lambda X}(x,1/2 ),\frac{cs-2\epsilon}{s + 2\epsilon},s + 2\epsilon \right) \le n, \qquad \text{for all } \lambda\geq \lambda_\eps \text{ and } s\in (0,  s_0). 
 \end{equation*} 
%  \begin{equation}%\label{eq:dimforballs}
%   \dim_N\left(B(x,r),\frac{cs-2\epsilon}{s + 2\epsilon},r(s + 2\epsilon)\right) \le n\qquad \text{for all } \lambda\geq \lambda_\eps \text{ and } s\in (0,  s_0). 
% \end{equation}
 With respect to the distance of $X$, the equation reads as
   \begin{equation*}%\label{eq:dimforballs}
   \dim_N\left(B_X(x,\dfrac{1}{2\lambda} ),\frac{cs-2\epsilon}{s + 2\epsilon},  \dfrac{  1}{\lambda} (s + 2\epsilon)\right) \le n,\qquad \text{for all } \lambda\geq \lambda_\eps \text{ and } s\in (0,  s_0),
 \end{equation*}
 which is \eqref{eq:dimforballs} with $r_\eps:= \dfrac{1}{\lambda_\eps}$, (put $r=\dfrac{1}{\lambda}$).
 By Lemma~\ref{lma:nagatafinalconclusion}, the first case then follows.
 
 Regarding the second case, we have that
 \begin{equation} \label{eq:unifdimfortan_bis}
  \dim_N(B_{T_xX}(\star, 1),c_x,s) \le n , 
   \end{equation}
 for some $c_x$ depending on $x\in K$ and for all $s>0$, since unique tangents admit dilations.
 Take now $x \in K$ and let $\eps_x \in (0,\frac18)$ be the $\eps$ in Lemma~\ref{lma:nagatafinalconclusion} given
 by the constants $s_0=\infty$, $a=4$, $b=12$, and
 $c=c_x$.
% (that will be later defined independently from $\epsilon$
 % {(Should we write the constants here?) }) with now $c$ replaced by $c_x$. 
 We abbreviate $\lambda_x := \lambda_{\eps_x}$. From \eqref{eq:tangentivicini} and \eqref{eq:unifdimfortan_bis}, Lemma~\ref{lma:dimfortangents} implies that,   since $1/2 < \min\{1/\eps_x , 1-2\eps_x\}$, 
   \begin{equation*}%\label{eq:dimforballs}
   \dim_N\left(B_{\lambda X}(x,1/2),\frac{c_xs-2\epsilon_x}{s + 2\epsilon_x},s + 2\epsilon_x \right) \le n \qquad \text{for all } \lambda\geq \lambda_x \text{ and all } s>0. 
 \end{equation*}
 So
    \begin{equation*}%\label{eq:dimforballs}
   \dim_N\left(B_X(x,\dfrac{1}{2\lambda} ),\frac{c_xs-2\epsilon_x}{s + 2\epsilon_x},  \dfrac{  1}{\lambda} (s + 2\epsilon_x)\right) \le n\qquad \text{for all } \lambda\geq \lambda_x	 \text{ and all } s>0.
 \end{equation*}

%{ (The constant $c$ depends on $\epsilon$. The use of compactness should be done only once we know which $\epsilon$ is needed for each ball.)} 
  Now take $y\in B_X(x,\frac{1}{4\lambda_x} )$ so that
  $ B_X(y,\frac{1}{4\lambda_x}) \subseteq  B_X(x,\frac{1}{2\lambda_x} )$
 and hence
 $$   \dim_N\left(B_X(y,\dfrac{1}{4\lambda_x} ),\frac{c_xs-2\epsilon_x}{s + 2\epsilon_x},  \dfrac{  1}{\lambda_x} (s + 2\epsilon_x)\right) \le n\qquad \text{for     all } s>0.$$
 Multiplying the distance by 
$  {4\lambda_x}$, we get
    \begin{equation}\label{eq:dimforballsBis}
 \dim_N\left(B_{ 4\lambda_x X}(y,1 ),\frac{c_xs-2\epsilon_x}{s + 2\epsilon_x},  4 (s + 2\epsilon_x)\right) \le n\qquad \text{for all  } s>0.
  \end{equation}
 Take $r\in (0,({4\lambda_x})^{-1})$. We need to compare $B_{  {4\lambda_x}X}(y,1 )$ and
 $B_{ \frac{1}{  r}X}(y,1 )$. Now we are going to use uniqueness of the tangents. Indeed, by the triangle inequality with $T_yX$, from \eqref{eq:tangentivicini}  we have
  \begin{equation}\label{eq:tangentivicini2}
 {\rm Dist}_{GH}\left((   {4\lambda_x} X,y) ,( \dfrac{1}{  r} X,y) \right)<2\eps_x,
 \end{equation}
 since $1/r\geq 4 \lambda_x\geq\lambda_x$.
 %  is small enough so that $\dfrac{2\lambda_\eps}{\tilde r} >\lambda_\eps$, (in fact, $\tilde r \leq 1-\eps <2$).
 Again from Lemma~\ref{lma:dimfortangents}, 
  by \eqref{eq:tangentivicini2} and \eqref{eq:dimforballsBis},  since $1/2 < \min\{1/(2\eps_x) , 1-4\eps_x\}$,
    \begin{equation*}%\label{eq:dimforballs}
   \dim_N\left(B_{\frac{1}{r} X}(y,1/2 ),   \frac{\dfrac{c_x   s-2\epsilon_x}{s + 2\epsilon_x} 4(s + 2\epsilon_x)-4\eps_x  }{4(s + 2\epsilon_x)+4\eps_x },4(s + 2\epsilon_x)+4\eps_x \right) \le n, % \qquad \text{for all } \lambda\geq \lambda_\eps \text{ and all } s>0. 
 \end{equation*}
 i.e.,
     \begin{equation*}%\label{eq:dimforballs}
   \dim_N\left(B_{\frac{1}{r} X}(y,1/2 ),   \frac{4 c_x    s- 12\epsilon_x}{4s +12\eps_x  }
  , 4s +12\eps_x\right) \le n. % \qquad \text{for all } \lambda\geq \lambda_\eps \text{ and all } s>0. 
 \end{equation*}
 Finally, multiplying the distance by $r$, we get  \eqref{eq:dimforballs} 
 for $x$ replaced by any $y\in B_X(x,\frac{ 1}{4\lambda_x}) $ and
 with $r_\eps:= \dfrac{1}{4\lambda_x}$, $s_0=\infty$, $a =4$, $b= 12$, and   $c=c_x$.
 By Lemma~\ref{lma:nagatafinalconclusion} we then have
 \[
  \dim_{LC}(B_X(x,\frac{ 1}{4\lambda_x})) \le n.
 \]
 By compactness of $K$ there exists a finite collection of balls $\{B_X(x,\frac{ 1}{4\lambda_x})\}_x$ covering $K$. 
 Using Corollary \ref{cor:Npart} we then conclude that
 \[
  \dim_{LC} K \le n,
 \]
 which finishes the proof of the second case.
%  This together with \eqref{eq:unifdimfortan} via Lemma~\ref{lma:dimfortangents} implies
% ------------------old part-----
% 
% Let $s_1 ,c \in (0,1)$ so that
% \begin{equation} %\label{eq:unifdimfortan}
%   \dim_N(T_xX,c,s) \le n \qquad \text{for all }x \in X \text{ and }0 < s < s_1.
% \end{equation} 
% Take $\epsilon \in (0,\frac{1-s_1}4)$. From the assumption that $\delta(\lambda,x) \searrow 0$ as $\lambda \to \infty$ uniformly in $x$, we
% have a radius $r_0>0$ so that 
% \[
%  d_{GH}((B(x,r),d/r,x),T_xX) \le \epsilon 
% \]
% for all $0 < r < r_0$. This together with \eqref{eq:unifdimfortan} via Lemma~\ref{lma:dimfortangents} implies
% \begin{equation}%\label{eq:dimforballs}
%   \dim_N\left(B(x,r),\frac{cs-2\epsilon}{s + 2\epsilon},r(s + 2\epsilon)\right) \le n
% \end{equation}
% for all $0 < s < s_1$, $x \in X$ and $0 < r < r_0$.
% 
% 
% ------------check the rest-------
\qed\\

\noindent
\textbf{Proof of Theorem~\ref{teorema1}.}
 Since $Y$ is relatively compact, the set $K := \cl{Y}$ is compact.
Note that $\dim_NK = \dim_NY$. By Theorem~\ref{thm:uniformtangents} with assumption (ii),
\[
 \dim_NY \le \sup\{\dim_NT_xX\,:\,x \in K\}.
\]
Take now a point $x$ in the interior of $Y$.
Then by Corollary \ref{cor:Nagatalower} $\dim_NT_xY \le \dim_NY$.
\qed

\subsection{Assouad dimension and uniformly close tangents}\label{Sec:Assouad_uniform_tangents}

 %Recall that, from Definition \ref{tangent spaces}, we denote by $ {\rm Tan}(X,x)$ the set of all tangents of $(X,d)$ at $x$, which are pointed metric spaces.
Let us now state the analog of  the first part of Theorem~\ref{thm:uniformtangents} for Assouad dimension.
%Due to the definition of Assouad dimension, the assumption on uniform constants now leads
%to a bit more complicated looking statement.

\begin{theorem}\label{thm:assouad1}
Let $X$ be a metric space with uniformly close tangents. Let $C > 1$ and $\alpha \ge 0$.
Assume that, for all $x\in X$ and for all $(Y, y)\in {\rm Tan}(X,x)$, the Assouad dimension
of $Y$ is less than or equal to $\alpha$ with constant $C$.
Then there exists $R > 0$ such that
\[
 \dim_A(B(x,R)) \le \alpha, \qquad \text{for all }x \in X.
\]
\end{theorem}

In fact, Theorem~\ref{thm:assouad1} is an immediate consequence of the following proposition where
we only need to have a cover of the balls in the tangents centered at the base point.

\begin{proposition}\label{prop:assouad1}
Let $X$ be a metric space with uniformly close tangents.
Assume that there are constants $C>1$ and $\alpha\geq0 $ such that,
for all $x\in X$, for all $(Y, y)\in {\rm Tan}(X,x)$, and for all $\delta \in (0,1)$,
we need at most $C(2/\delta)^\alpha$ balls of radius $\delta/2$ to cover the ball $B(y,2)$.
Then there exists a constant $R'>0$ such that, for all $\beta >\alpha$, there exists $C'>1$
for which
$$\dim_A ( B(x,R') , C', R' )\leq \beta , \qquad \text{for all }x \in X.$$
% such that
%for all $r,R$ such that $0<r<R<R'$ and for all $x\in X$, the ball $B(x,R)$ can be covered with 
 %less than $C'(R/r)^\beta$ balls of radius $ r$.
\end{proposition} 

The analog of the second part of Theorem~\ref{thm:uniformtangents} is the following.

\begin{theorem}\label{thm:assouad2}
Let $X$ be a metric space with uniformly close unique tangents in a compact subset $K \subset X$.
 Then
\[
  \dim_AK \le \sup\{\dim_AY \,:\, x \in K, Y \in {\rm Tan}(X,x)\}.
 \]
\end{theorem} 

Again, as for the Nagata dimension, we have to make one of the two assumptions on the tangents, given in the two previous theorems, 
in order to be able to get a bound on the Assouad dimension of the space from the Assouad dimensions of its tangents.
The following example, similar to Example \ref{ex:nagata}, shows that the assumptions are necessary.

\begin{example}\label{ex:assouad}
 In Example \ref{ex:nagata} we defined the space $X$ as a subset of the Euclidean plane.
 The construction had three stages: first we defined $S_n$, then using it $E_n$ and finally $X$.
 Here we replace $S_n$ by a set $\{0,1\}^n$. We want the points in this set to be equidistant
 and so we consider $\{0,1\}^n \subset \R^n$ with the maximum norm.

 We define $E_n \subset \R^n$ as
 \[
  E_n = \left\{\sum_{i=n}^\infty a_i 2^{-i^2} \,:\, a_i \in \{0,1\}^n \text{ if } i \text{ odd, and }a_i \in \{0,1\}\text{ if }i\text{ even }\right\}
 \]
  and using it we define $X \subset \R^\N$ as
 \[
  X = \cl\left({\bigcup_{n=1}^\infty (E_n + 2^{-n^2})} \right),
 \]
 where the embedding of different dimensional $\R^n$ to $\R^\N$ are understood 
 by identifying $x \in \R^n$ with $(x,0,0,\dots) \in \R^\N$. We equip $X$ with the supremum distance of $\R^\N$.
 
%Figure \ref{fig:example2} contains an illustration of the space $X$.
%\begin{figure}\label{fig:example2}
%  \centering
%  \psfrag{e1}{$E_1$}
%  \psfrag{e2}{$E_2$}
%  \psfrag{e3}{$E_3$}
%  \includegraphics[scale=.7]{uniftang}
%  \caption{
%    Illustration of the construction of a compact space with Assouad dimension infinity 
%    that has uniformly close tangents, all with Assouad dimension zero.
%  }
%\end{figure}

 We have
 \[
  {\rm Tan}(E_n,x) = \left\{\{0,t\}^n \,:\, t \in [0,\infty)\right\} \cup \left\{\{0,t\} \,:\, t \in [0,\infty)\right\}
 \]
 for all $n \in \N$ and $x \in E_n$, and
 \[
  {\rm Tan}({X},0) = \left\{\{0,t\} \,:\, t \in [0,\infty)\right\}.
 \]

As in Example \ref{ex:nagata}, we have uniformly close tangents essentially because 
\[
 \left\{\{0,t\} \,:\, t \in [0,\infty)\right\} \subset {\rm Tan}(E_n,x)
\]
for all $x \in X$. This set of tangents takes care of the scales where other sets of the form $\{0,1\}^n$ 
can not be yet seen.

Take any $\alpha > 0$.
For any  $x \in X$ and any $Y \in {\rm Tan}(X,x)$ the tangent $Y$ has only finite number of points and hence $\dim_AY \le \alpha$.
However, for any $n \in N$ the set $\{0,1\}^n$ needs $2^n$ balls of radius $r < 1$ to cover it. Therefore
for any $C>0$ there are arbitrarily small balls $B(x,r)$ in $X$ that can not be covered by less than $C2^\alpha$ balls of radius $\frac{r}2$.
Thus $\dim_AX \ge \alpha$. This was true for any $\alpha>0$ and so,
\[
 \dim_AX = \infty \qquad \text{and} \qquad \sup\{\dim_AY \,:\, x \in X, Y \in {\rm Tan}(X,x)\} = 0.
\]
Hence, the two assumptions in Theorem~\ref{thm:assouad1} and in Theorem~\ref{thm:assouad2} are necessary.
\end{example}

Before proving Proposition \ref{prop:assouad1} we provide a lemma.

\begin{lemma}\label{lem_ass_1}
Assume that, for some $\delta\in (0,1)$, we have
${\rm Dist}_{GH}(( X,x) ,( Y,y) )< \delta/4 $
and that, for some $L\in \N$, we have that the ball
$B_Y(y,2)$ can be covered with  $L$ balls of radius $\delta/2$.
Then  the ball
$B_X(x,1)$ can be covered with $L$ balls of radius $\delta$.
\end{lemma}

\proof
 Set $\eps=\delta/4$.
Since $( X,x) $ and $( Y,y) $ have distance $< \epsilon$, we can see them as a subset of a metric space $Z$ such that $	d(x,y)\leq \epsilon,
	B_{  X}(x, 1/\epsilon) \subseteq B_{Z} ( Y, \epsilon),$ and $
	B_{Y}(y, 1/\epsilon) \subseteq B_{Z} (   X, \epsilon)$.
	
By assumption, there are points $y_1, \ldots, y_L\in Y$ such that
$$B_Y(y,2)\subseteq \bigcup_{j=1}^L B_Y(y_j, \delta/2).$$
We may assume that
$d(y,y_j)\leq 2+\delta/2.$
Since $2+\delta/2<1/\eps$, for all $j$ there is $x_j\in X$ with $d_Z(x_j, y_j)<\eps.$
We claim that 
$$B_X(x,1)\subseteq \bigcup_{j=1}^L B_X(x_j, \delta).$$
Indeed, pick $x'\in B_X(x,1)$. Then there is $y'\in B_Y(y, 1+2\eps)$ such that 
$d_Z(x', y')<\eps.$
Since $1+2\eps<2$, there exists $j$ such that $y'\in B_Y(y_j, \delta/2).$
Therefore, 
$d(x',x_j)\leq d(x',y')+d(y',y_j)+d(y_j,x_j)<\eps+\delta/2+\eps=\delta.$
\qed
\\
\\
{\bf Proof of Proposition \ref{prop:assouad1}.}
 Fix  $\beta >\alpha$.
Let $\delta\in (0,1)$ be such that
$(1/\delta)^\beta = C (4/\delta)^\alpha=:L.$
By assumption, there exists $\lambda_0>1$ such that, for all $\lambda >\lambda_0$ and for all $x\in X$, there exists   
${(Y,y)} \in {\rm Tan}(X,x)$  such 
$$ {\rm Dist}_{GH}( (\lambda X,x), {(Y,y)}) <\delta /4. $$

Let $R$ and $r$ such that $0<r<R<1/\lambda_0$.
Fix $N$ such that
$$\delta^N\leq \dfrac{r}{R} \leq \delta^{N-1}.$$
We will apply Lemma~\ref{lem_ass_1} $N$ times.
First, since $1/R>\lambda_0$,  by  Lemma~\ref{lem_ass_1}, we have that,  in the contracted metric space $\frac{1}{R}X$,
we need at most  $L:=C (4/\delta)^\alpha$ balls of radius $\delta$ to cover 
 $B_{\frac{1}{R}X} (x,1)$. Such balls are of the form
 $$B_{\frac{1}{R}X} (x',\delta) =B_{\frac{1}{\delta R}X} (x',1)   .$$
 Each of these balls needs at most $L$ 
balls of radius $\delta$ to cover it in the space  $\frac{1}{\delta R}X$.
Iterating, we need at most $L^N$ balls of radius $\delta$, with respect to 
the distance of  $\dfrac{1}{\delta^{N-1} R}X$, to cover
 $$B_{\frac{1}{R}X} (x,1) =B_{ X} (x,R).$$
 Such sets are balls of radius 
$\delta^{N} R$ 
 for $X$.
 In other words, we covered 
 $B_{ X} (x,R)  $ with balls of radius $r$, since $r> \delta^{N} R$.
 The number of these balls is bounded by 
 $$L^N=  (1/\delta)^{\beta N}
 =(1/\delta)^\beta (1/\delta^{N-1})^\beta   
\leq
 1/\delta^\beta (R/r)^\beta   
 .$$
 Applying Lemma  \ref{lem_ass_2}
  with $C_\beta :=  1/\delta^\beta $, we conclude.
 \qed
\\
\\
{\bf Proof of Theorem~\ref{thm:assouad2}.}
 Denote by $(T_xX,\star)$ the unique element in ${\rm Tan}(X,x)$, for $x \in K$. Take 
 \[
 \beta > \alpha > \sup\{\dim_AT_xX \,:\, x \in K\}.%, {\rm Tan}(X,x) = \{T_xX\}\}.
  \]
 Fix some $x \in K$ and let $1 < C_x < \infty$ be such that for any $0 < r < R < \infty$ and $y \in T_xX$
 we need at most
 $C_x(R/r)^\alpha$ balls of radius $r$ to cover the
 ball $B_{T_xX}(y,R)$. Let $\delta_x\in (0,1)$ be such that
 $(\frac1{16\delta_x})^\beta = C_x (4/\delta_x)^\alpha=:L.$
By assumption, there exists $\lambda_x>1$ such that, for all $\lambda >\lambda_x$ and for all $y\in K$ we have
\[
 {\rm Dist}_{GH}( (\lambda X,y), (T_yX,\star)) <\delta_x /4.
\]

Our aim is to show that for any $z \in B(x,\frac1{4\lambda_x}) \cap K$ and $0 < r < \frac1{4\lambda_x}$
we need at most $L$ balls of radius $8\delta_x r$ to cover the ball $B(z,r)$. Fix such $z$ and $r$. By the definition of $C_x$ the ball 
$B_{T_xX}(\star,2)$ needs at most $L$ balls of radius $\frac{\delta_x}{2}$ to cover it. Since $2\lambda_x > \lambda_x$ 
we have
\[
 {\rm Dist}_{GH}( (2\lambda_x X,x), (T_xX,\star)) < \delta_x / 4,
\]
and so by Lemma~\ref{lem_ass_1} we need at most $L$ balls
of radius $\frac{\delta_x}{2\lambda_x}$ to cover the ball $B(x,\frac1{2\lambda_x})$.
Since $B(z, \frac1{4\lambda_x}) \subset B(x,\frac1{2\lambda_x})$, the same collection of balls covers $B(z, \frac1{4\lambda_x})$.

Using the fact that the tangent of $X$ at $z$ is unique and the triangle inequality, we get
\begin{align*}
 {\rm Dist}_{GH}( & (r^{-1} X,z), (8\lambda_xX,z)) \\
 & \le {\rm Dist}_{GH}( (r^{-1} X,z), (T_zX,\star)) + {\rm Dist}_{GH}( (T_zX,\star), (8\lambda_xX,z)) < \delta_x /2. 
\end{align*}

Since the ball $B_{8\lambda_xX}(z,2) = B_X(z,\frac1{4\lambda_x})$ needed at most $L$ balls of the type 
$B_{8\lambda_xX}(x',4\delta_x) = B_X(x',\frac{\delta_x}{2\lambda_x})$ to cover it, by Lemma~\ref{lem_ass_1} we need no more than
$L$ balls of radius $8\delta_x r$ to cover the ball $B(z,r)$.

Now we continue like in the proof of Proposition \ref{prop:assouad1}.
Let $R$ and $r$ such that $0<r<R<\frac1{4\lambda_x}$.
Fix $N$ such that
$$(16\delta_x)^N\leq \dfrac{r}{R} \leq (16\delta_x)^{N-1}.$$

First, since $R< \frac1{4\lambda_x}$,  by the above considerations 
we need at most $L$ balls of radius $8\delta_x R$ to cover the set $B(x,R)\cap K$, and hence at most $L$ balls of radius $16\delta_x R$ centered at $B(x,R) \cap K$.
Each of these balls, since they are centered at $B(x,R)\cap K$, need at most $L$ balls of radius $(16\delta_x)^2 R$ centered at $B(x,R) \cap K$ to cover them.
Continuing inductively $N$ times, we obtain a cover of $B(x,R)\cap K$ with at most $L^N$ balls of radius $(16\delta_x)^N R$.
Again we estimate the number 
\[
L^N = (16\delta_x)^{-\beta N} =(16\delta_x)^{-\beta} ((16\delta_x)^{1-N})^\beta   
\leq
 (16\delta)^{-\beta} (R/r)^\beta.
\]
Therefore, $\dim_A(B(x,\frac1{4\lambda_x})\cap K) \le \beta$. By compactness of $K$ we then have
 $\dim_A K \le \beta$. Letting $\alpha$ and $\beta$ tend to $\sup\{\dim_AT_xX \,:\, x \in K, {\rm Tan}(X,x) = \{T_xX\}\}$
 finishes the proof.
\qed\\

\noindent
\textbf{Proof of Theorem~\ref{teorema1b}.}
 Since $Y$ is relatively compact, $K := \cl{Y}$ is compact.
Note that $\dim_AK = \dim_AY$. By Theorem~\ref{thm:assouad2},
\[
 \dim_AY \le \sup\{\dim_AT_xX\,:\,x \in K\}.
\]
Take now a point $x$ in the interior of $Y$.
Then by Corollary \ref{cor:Assouadlower} $\dim_AT_xY \le \dim_AY$.
\qed

\section{Nagata dimension of Carnot groups and   equiregular subRiemannian manifolds}\label{sec:equiSR}
Let $M$ be a smooth manifold. Let 
$\distr$ be a smooth subbundle %(a.k.a. vector distribution of constant rank) 
of the tangent bundle of $M$.
Denote by $\Gamma(\distr)\subset \VecM$   the $C^{\infty}(M)$-module of the smooth sections of $\distr$. 
One says that  $ \distr$ satisfies the \emph{bracket-generating condition} if
$$ \bigcup_{ j\in \N} \distr_q^j=T_qM, \quad \forall q\in M,
$$
where
\begin{equation} \label{Hor}
\distr_q^j:= \text{span}\{[X_1,[X_2,\ldots,[X_{j-1},X_j]]](q)~|~X_i\in\Gamma(\distr)\} \subseteq T_qM, \quad \forall q\in M, \, j\in \N .
 \end{equation}

A subRiemannian manifold  is a triple $ (M,\distr,\metr)$, where $M$ is a connected   smooth manifold, $\distr$ is a bracket-generating subbundle %(a.k.a. vector distribution of constant rank) 
of the tangent bundle of $M$
 %that satisfies the H\"ormander condition 
 and $\metr$ is a Riemannian metric tensor restricted to $\distr$. 
% \begin{note}
%The
%bracket generating condition 
% is also know as 
%H\"ormander condition and the  subbundle is also known as horizontal (vector) distribution of constant rank.\end{note}

A subRiemannian manifold has a natural structure of metric space, where the distance is the so-called {\em Carnot-Carath\'eodory distance}
\begin{align*}
d_{cc}(p,q)=
\inf\Big\{\int_0^T&\sqrt{\metr_{\gamma(t)}(\dot\gamma(t),\dot\gamma(t))}\,dt~\Big|~ \gamma:[0,T]\to M \mbox{ is a Lipschitz curve},\\
&\quad\qquad\qquad\qquad\gamma(0)=p,\gamma(T)=q, ~~\dot \gamma(t)\in\distr_{\gamma(t)}\mbox{ a.e. in $[0,T]$} \Big\}.
\end{align*}
As a consequence of Chow-Rashevsky
Theorem
 such a  distance is always finite  and induces on $M$ the original topology.
%Since $(M,d)$  is a metric space, for every $\alpha>0$ one can define the $\alpha$-dimensional Hausdorff measure on $M$, and compute the Hausdorff dimension of $M$.

% 
% The \emph{flag} of $\distr$ at $q\in M$ is the sequence of vector spaces $\distr_q^0\subset\distr_q^{1}\subset\distr_q^{2}\subset\ldots\subset T_qM$ defined by
%\begin{equation*}
%\distr^0_q:= \{0\},\qquad\distr_q^{1}:=\distr_q,\qquad\distr_q^{i+1}:=\distr_q^{i}+[\distr^{i},\distr]_q\,,
%\end{equation*}
%where, with a standard abuse of notation, we understand that $[\distr^i,\distr]_q$ is the vector space generated by the iterated Lie brackets, up to length $i$, of local sections of the distribution, evaluated at $q$.
% 
%A distribution $\distr$ is \emph{equiregular} if, for each $i=1,2,\ldots,m$, $ \dim (\distr^{i}_q)$ does not depend on $q\in M$.  
% 
% OR

 The subRiemannian manifold is called  {\em equiregular} if
%defining $\distr^{1}:=\distr, \distr^{i+1}:=\distr^{i}+[\distr^{i},\distr]$, for every $i=1,2,\ldots$,
 the dimensions of the spaces $\distr^{i}_q,\  i\in \N$, as defined in \eqref{Hor}, do not depend on the point $q$.

% 
% --
% 
% 
% {\color{green}
% Extra: 
% In this case, the H\"ormander condition guarantees that there exists (a mimimal) $m\in \N$, called \emph{step} of the structure, such that $\distr_{q}^{m}=T_{q}M$, for all $q\in M$. 
% 
%
%$$\text{gr}_{q}(\distr)=\bigoplus_{i=1}^{m} \distr^{i}_{q}/\distr^{i-1}_{q}, \qquad\text{where} \ \  \distr_{q}^{0}=0.$$
%
% 
% 
%Finally, we introduce the nilpotentization of an equiregular distribution at the point $q$,  
%  The \emph{nilpotentization of $\distr$} at the point $q\in M$ is the graded vector space
%\begin{equation*}
%\text{gr}_q(\distr) = \distr_q \oplus \distr_q^2/\distr_q \oplus \ldots \oplus \distr_q^m/\distr_q^{m-1}.
%\end{equation*}
% 
%
%The vector space $\text{gr}_q(\distr)$ can be endowed with a Lie algebra structure, which respects the grading. Then, there is a unique connected, simply connected group, $\text{Gr}_q(\distr)$, such that its Lie algebra is $\text{gr}_q(\distr)$. The global, left-invariant vector fields obtained by the group action on any orthonormal basis of $\distr_q\subset \text{gr}_q(\distr)$ defines a sub-Riemannian structure on $\text{Gr}_q(\distr)$, which is called the \emph{nilpotent approximation} of the sub-Riemannian structure at the point $q$.
%The space $\text{Gr}_q(\distr)$ is an example of Carnot group.
%}
%
%----

Very particular subRiemannian manifolds are the Carnot groups.
Let us briefly recall that a {\em Carnot group} is a stratified Lie group endowed with a left-invariant 
Carnot-Carath\'eodory metric where the subbundle is the first 
stratum of the Lie algebra.
See \cite{LeDonne_characterization} for an introduction to Carnot groups from a point of view of metric geometry.

\begin{theorem}[Mitchell-Bella\"iche \cite{bellaiche}]\label{Mitchell:thm}
Let $M$ be an equiregular subRiemannian manifold equipped with its Carnot-Carath\'eodory distance $d_{cc}$.
Then, at every point the tangent space is a Carnot group of same topological
dimension. Moreover,
the convergence to the tangents is uniform on compact sets.
\end{theorem}

We will be able to deduce the local Nagata dimension of an equiregular subRiemannian manifold, since we know the Nagata dimension of Carnot groups,
by a result 
 %The result was 
 originally proved by 
Urs Lang and the first-named author.  
Many thanks go to Lang for giving the permission to include here a short proof. % \ref{Lang_LeDonne:proof}.
\begin{theorem}[Lang \& Le Donne]\label{Lang_LeDonne:thm}
The Nagata dimension of a Carnot group equals its topological dimension.
\end{theorem}

%\noindent
%\textbf{Proof of Theorem~\ref{Lang_LeDonne:thm}}
\proof
Let $G$ be a Carnot group of topological dimension $n$. 
Let $d_C$ be the Carnot-Carath\'eodory distance, 
and let $d_R$ be a left-invariant Riemannian distance. 
It follows from Lemma~\ref{local} (see Example \ref{ex:Liegroup}) that $(G,d_R)$ has linearly controlled dimension $n$, 
with parameters $c,\bar s$, say.
%(compare~\cite[Corollary 4.5]{HiP}).
Since $d_C$ and $d_R$ give the same topology on $G$, we can fix a scale 
$s_0 \in (0,\bar s]$ such that $B^{d_R}(e,cs_0) \subseteq B^{d_C}(e,1)$, where $e$ 
is the identity element of $G$; furthermore there exists $\rho > 0$ such that 
$B^{d_C}(e,\rho) \subseteq B^{d_R}(e,s_0)$. Since both $d_C$ and $d_R$ are 
left-invariant, this means that $d_R(p,q) \le cs_0$ implies $d_C(p,q) \le 1$,
and $d_C(p,q) \le \rho$ implies $d_R(p,q) \le s_0$.
Now let $\mathcal{B} $ be a $cs_0$-bounded cover, with $s_0$-multiplicity at 
most $n+1$, of $(G,d_R)$. With respect to $d_C$, $\mathcal{B} $ is $1$-bounded and 
has $\rho$-multiplicity at most $n+1$. For every $\lam > 0$, there exists a 
dilation of $(G,d_C)$ by the factor $\lam$, that is, a bijection 
$\del_\lam \colon G \to G$ such that $d(\del_\lam(p),\del_\lam(q)) 
= \lam d(p,q)$ for all $p,q \in G$.
Applying $\del_\lam$ to the members of $\mathcal{B} $, we obtain a $\lam$-bounded 
cover of $(G,d_C)$ with $\lam \rho$-multiplicity at most $n+1$.
Thus $(G,d_C)$ has Nagata dimension at most $n$, with
constant $1/\rho$. On the other hand, from \cite[Theorem 2.2]{Lang-Schlichenmaier} we have $\dim_N(G,d_C)\geq n$. 
\qed
\\

% If the subRiemannian structure is given by a regular  distribution $\Delta\subset T M$

\noindent
\textbf{Proof of Corollary \ref{corollarioSRM}.}
Carnot-Carath\'eodory distances are locally doubling (this fact is proved in \cite{nagelstwe}, however, notice that Theorem~\ref{teorema1b} gives an alternative proof).
From \cite{Lang-Schlichenmaier} (or from Theorem~\ref{thm:nagaasso}), we deduce that the Nagata dimension of a compact set of a subRiemannian manifold is finite.
Hence,  from Theorem~\ref{teorema1}, together with  Theorem~\ref{Mitchell:thm} and Theorem~\ref{Lang_LeDonne:thm}, we have Corollary \ref{corollarioSRM}.
\qed
%\begin{corollary}
% Let  $(M,d_{cc})$ be an equiregular subRiemannian manifold. %Let $p \in M$ be a regular point for the horizontal   distribution.
%  Then the  Nagata   dimension  of any open bounded nonempty subset of $M$ equals   the topological dimension of the manifold.
%\end{corollary}

\begin{remark}
In the case that the subRiemannian manifold is not equiregular, one can still use the fact that
the set of regular points for the subbundle is an open dense subset of the manifold.
Here, a point $p$ is called {\em regular} if it has a neighborhood on which the   subbundle is equiregular. Hence, we have that 
 there exists a neighborhood of $p$ whose Nagata dimension  equals   the topological dimension of the manifold.
 \end{remark}

 \begin{remark}
 Probably  the most famous  Carnot-Carath\'eodory space that is not  an equiregular subRiemannian manifold is the Grushin plane.
 In this example the dimension of $\Delta_q$ varies with $q$.
 Also in this case the Nagata dimension is equal to the topological one. Indeed, recently Meyerson 
 gave simple examples of  quasi-symmetric maps between the Grushin plane and the Euclidean plane. 
 In the usual coordinates, an example of such  maps is $F(x, y) = (x|x|, y)$, see \cite{Meyerson}.
  On the other hand, Lang and Schlichenmaier in \cite{Lang-Schlichenmaier} proved that the Nagata dimension is a quasi-symmetric invariant.
  Therefore,  the Nagata dimension of the Grushin plane is two, like the topological dimension.
  \end{remark}

On compact sets, a Carnot-Carath\'eodory distance $d_C$ and a Riemannian distance $d_R$ on $G$ satisfy the inequalities
\begin{equation}\label{partial:snow}
 k \,(d_C)^m \leq d_R\leq d_C,
 \end{equation}
 for some  $m\in \N $ and $k\in(0,1)$, see \cite{nagelstwe}.
 Because snowflaking a distance preserves the Nagata dimension, it is natural to wonder if partial snowflaking as in \eqref{partial:snow} preserves it as well, giving an alternative method of proof for Theorem~\ref{Lang_LeDonne:thm} and  
  Corollary \ref{corollarioSRM}.
 We give an example showing that this is not the case.

\begin{example}
We introduce  a topological space $X$ with two proper distances $d_1$ and $d_2$ with the property that
$$d_1^2 \leq d_2\leq d_1,$$
but $\dim_N(X,d_1)\neq\dim_N(X,d_2) $.

The space $X$ is of the form $\{p_{n,j} :{n\in\N,j=0,\ldots,n }\}$ endowed with the discrete topology.
The distances are 
$$d_1(p_{n,j},p_{n',j'})= \left\{\begin{array}{ccl}\left|\dfrac{1}{2^n}-\dfrac{1}{2^{n'}}\right|,&\qquad& \text{if } n\neq n',\\
\dfrac{1}{2^n},&\qquad &\text{if }  n= n', j\neq j'\\
0,&\qquad &\text{if }  n= n', j= j'
\end{array}\right.$$
and
$$d_2(p_{n,j},p_{n',j'})= \left\{\begin{array}{ccl}\left|\dfrac{1}{2^n}-\dfrac{1}{2^{n'}}\right|,&\qquad& \text{if } n\neq n',\\
\dfrac{|j-j'|}{n2^n},&\qquad &\text{if }  n= n'.
\end{array}\right.$$
One can easily check that these are distance functions, that satisfy $d_1^2 \leq d_2\leq d_1,$ and that 
$\dim_N(X,d_1)=0$ but $\dim_N(X,d_2) \neq0$. The only calculation that is not completely straightforward is that $\dim_N(X,d_2) \neq0$. For doing this, one assumes that the dimension is $0$ with respect to some constant $c$. Then we fix $n\in \N$ and let $s=\dfrac{1}{n2^n}$. Take a    $cs$-bounded cover with $s$-multiplicity
at most $  1$. This last property implies that $p_{n,0}$ and $p_{n,1}$, which have distance $s$, need to be contained in the same element $U$ of the cover. Likewise, $p_{n,j}\in U$, for all $j=0,\ldots, n$.
Thus $\dfrac{1}{2^n}\leq \diam U\leq c s = c \dfrac{1}{n2^n}.$
So $c>n$, for any $n\in \N$, a contradiction.
 \end{example}

\section{Assouad dimension bounds Nagata dimension}\label{Sec:nagaasso}
%In this section we prove that $ \dim_NX  \le \dim_AX$, for all metric spaces $X$.

%In addition, we give the following result relating the two dimensions.
%We also give a bounded-scale version of this result %Theorem~\ref{thm:nagaasso}
%in Theorem~\ref{thm:nagaasso2} at the end of Section \ref{Sec:nagaasso}.

Denoting by $ \dim X $ the topological dimension of a metric space $X$ and by  $\dim_HX $ its Hausdorff dimension,
recall that 
one has the   chain of inequalities
\[
 \dim X \le \dim_HX   \le \dim_AX.
\]
From the bound proven in this section, we will conclude  that we also have
 the inequalities
\[
 \dim X \le \dim_NX  \le \dim_AX,
\]
where the first inequality is obtained in \cite{Lang-Schlichenmaier}.
In \cite{Lang-Schlichenmaier},  Lang and Schlichenmaier also proved  the inequality  $\dim_N X\leq 3^{\dim_A X}-1.$
We optimally improve it.
We should also point out that in general there is no relation between  $\dim_N $ and $\dim_H$. As examples,
 on one hand $\dim_N\Q = 1 > 0 = \dim_H\Q$, and on the other hand
 $\dim_N(\R,|\cdot|^\frac12) = 1 < 2 = \dim_H(\R,|\cdot|^\frac12)$. 
  %Notice that snowflaking leaves $\dim_N$ unchanged but increases $\dim_H$.
%One can easily observe that the same arguments also prove that
%the linearly controlled dimension is at most 
%the local Assouad dimension, at least for separable metric spaces.
After the proof, in Theorem~\ref{thm:nagaasso2} we show how the argument can be modified
to give a bounded-scale version of the same result.
\\

\noindent
{\bf Proof of Theorem~\ref{thm:nagaasso}.}
Let $X$ be a metric space.
We shall prove that $ \dim_NX  \le \dim_AX$. %, for all metric spaces $X$.
 Without loss of generality we may assume $\dim_AX < \infty$.
 Take $\alpha$ so that $\dim_AX < \alpha < \lfloor \dim_AX \rfloor + 1$. From the definition of the Assouad dimension we know that there is some
 constant $C>0$ such that, for any $x \in X$ and $0 < r < R < \infty$, 
 \begin{equation}\label{eq:Assouad}
  \text{we need at most }C\left(\frac{R}{r} \right)^\alpha \text{ balls of radius }r\text{ to cover the set } B(x,R)\cap X.
 \end{equation}

We shall fix $r\ll R$ to be determined later in terms of $C$ and $\alpha$ only.
The idea of the proof is the following.
First decompose $X=:X_0$ as
$$X_0=   X_1 \sqcup Y_1^1  \sqcup Y_2^1 \sqcup \cdots$$
with the properties that:
diam$(Y_n^1)<2 R$, $\dist(Y_n^1,Y_m^1)>r,$ for all $n\neq m$, and all balls 
$B(x,R)\cap X_1$, with $x$ in $X_1$, need at most
$C_1\left(\frac{R}{r} \right)^{\alpha-1}$ 
 balls of radius $r$ to cover them.
Here $C_1$ is a constant  depending only on $\alpha$ and $C$.
Hence, roughly speaking, the subset $X_1$ has codimension one on scales $R$ and $r$.
Then we will iterate the decomposition  $X_{k-1}= X_k    \sqcup Y_1^k  \sqcup Y_2^k \sqcup \cdots$
with similar properties for $Y_n^k$. As soon as $k>\alpha$ and both $R$ and $r$
are properly chosen, the subset $X_k$ is empty.
The collection $\{Y_n^k   :  {k=1},\ldots,    \lfloor \alpha \rfloor + 1 , \, n=1 ,\ldots,{N_k} \}$ gives a cover  showing that the Nagata dimension is less than  $\alpha$.

 Take $R >0$ and let $\{x_i\}_{i=1}^{N_1} \subset X$ be a maximal $\frac{R}4$-separated net of points.
 (It might be that $N_1 = \infty$.
 However, recall that metric spaces with finite Assouad dimension are separable and hence separated nets are countable.)

 Notice that in particular the 
 balls $B(x_i,\frac{R}2)$ cover the set $X$. Take $0< r < \frac{R}4$. 
 For each $x_n \in \{x_i\}$ there exists by \eqref{eq:Assouad} a collection $\mathcal{B} = \{B(y_i,r)\}$ of at most
 $C\left(\frac{R}{r} \right)^\alpha$ balls covering the larger ball $B(x_n,R)$. Let $k \in \N$ be such that
 $kr<\frac{R}2 \le (k+1)r$.
 Consider the annular regions
 \[
  A_{n,i} := B(x_n,R-ir) \setminus B(x_n,R-(i+1)r).
 \]
 The collection $\{A_{n,i}\}_{i=0}^{k-1}$ is disjointed and
 \[
  \bigcup_{i=0}^{k-1}A_{n,i} = B(x_n,R) \setminus B(x_n,R-kr) \subset B(x_n,R) \setminus B(x_n,R/2).
 \]
 Observe that any ball $B \in \mathcal{B}$ intersects at most $3$ of the annular regions $A_{n,i}$.

 Therefore there exists some $0 \le k_n \le k-1$ such that
 the annulus $A_{n,k_n}$ meets at most
 $  \frac3kC\left(\frac{R}{r}\right)^\alpha $ balls of the collection 
 $\mathcal{B}$.
 Notice that we can bound 
 \begin{equation*}%\label{eq:-1cover}
  \frac3kC\left(\frac{R}{r}\right)^\alpha \le \frac{3r}{\frac{R}{2}-r}C\left(\frac{R}{r}\right)^\alpha
  \le 12C\left(\frac{R}{r}\right)^{\alpha-1}  .
 \end{equation*} 
  %balls from the collection $\mathcal{B}$ cover $A_{n,k_n}$. 
 Then, in particular, 
 %\eqref{eq:-1cover} 
  \begin{equation}\label{eq:-1cover}
 \text{for all } n, \text{ the number of balls
 needed to cover } A_{n,k_n}\text{ is at most } 12C\left(\frac{R}{r}\right)^{\alpha-1} . 
 \end{equation}

 Set $X_1 := \bigcup_{n}A_{n,k_n}$. 
% \begin{note}
%  Idea behind is that $Y$ is $\alpha - 1$ dimensional when seen on the scales $R$ and $r$.
% \end{note}
 Now, take a point $x \in X_1$. 
 %Let $i \in \N$ be such that $x \in A_{i,k_i}$. 
  We claim that 
 \begin{equation}\label{eq:samesizeintersect}
 \text{ the ball } B(x,R) \text{  intersects at most } C16^\alpha \text{ annular regions } A_{i,k_i}.
 \end{equation}
 Indeed, the number of annular regions that $B(x,R)$ meets is less than the number of balls 
 in $\mathcal{B}$ that it meets.
 Hence, we need to estimate the number of elements of the net $\{x_i\}_i$ in $B(x,2R)$.
 By \eqref{eq:Assouad}, we need at most $C\left(\frac{2R}{R/8}\right)^\alpha$ balls of radius $R/8$ to cover the ball $B(x,R)$.
 Since the elements of the net are $R/4$ separated, any two of them are in different such balls of radius $R/8$. 
 Hence the number of  $x_i$'s inside $B(x,2R)$ is at most
 $$  C\left(\frac{2R}{R/8}\right)^\alpha = C16^\alpha.$$
 
% \begin{note}
%  Write down the details of how to obtain \eqref{eq:samesizeintersect}.
% \end{note}

 Therefore, by combining \eqref{eq:-1cover} and \eqref{eq:samesizeintersect}
 we see that we need at most $C_1 \left(\frac{R}{r}\right)^{\alpha-1}$
 %\[
%C_1 \left(\frac{R}{r}\right)^{\alpha-1}
% \]
 balls of radius $r$ to cover $X_1 \cap B(x,R)$, where $C_1:=  C^216^\alpha12$.

 Define, for all $n$, 
 \[
  Y_n^1 := B(x_n, R-(k_n+1)r) \setminus \bigg(X_1\cup\bigcup_{i=1}^{n-1} B(x_i, R-k_ir)\bigg).
 \]
 What we have now obtained is a decomposition of $X$ into disjointed collection $\{Y_n^1\}_{n=1}^{N_1} \cup \{X_1\}$ 
 with $Y_n^1$ having diameter less than $2R$ and the property that $\dist(Y_i^1,Y_j^1) \ge r$ for any $i \ne j$.

 Repeating the above argument we can decompose $X_1$ into disjointed collection $\{Y_n^2\}_{n=1}^{N_2} \cup \{X_2\}$
 with $Y_n^2$ again having diameter less than $2R$ and the property that $\dist(Y_i^2,Y_j^2) \ge r$ for any $i \ne j$,
 and $X_2$ such that for any point $x \in X_2$ we need at most $C_2\left(\frac{R}{r}\right)^{\alpha-2}$
 %\[
 % C_2\left(\frac{R}{r}\right)^{\alpha-2}
 %\]
 balls of radius $r$ to cover the set $X_2 \cap B(x,R)$, where $C_2$ depends only on $C$ and $\alpha$.

 We continue this for $m = \lfloor \alpha \rfloor + 1$ steps so that we have a decomposition of $X$ into
 disjointed collection
 \[
  \bigcup_{k=1}^m\{Y_n^k\}_{n=1}^{N_k}\cup\{ X_m\}
 \]
 with all $Y_n^k$ having diameter less than $2R$ and $\dist(Y_i^k,Y_j^k) \ge r$ for any $i \ne j$.
 But now at the last iteration, step $m$, when we have arrived to \eqref{eq:-1cover} we notice that there exists an annular region
 that intersects at most
 \[
  12C_{m-1}\left(\frac{R}{r}\right)^{\alpha-m}
 \]
 balls of the cover. Provided that we have chosen the ratio $\frac{r}{R}$ to be small enough from the beginning,
 this means that the annular region is empty, since $\alpha - m < 0$. The conclusion is that $X_m = \emptyset$.

Hence, we have $$ X= \bigcup_{k=1}^m\bigcup_{n=1}^{N_k} Y_n^k.$$
 Now, given any $x \in X$, for any $k=1, \ldots, m$,  the ball $B(x,r/2)$ intersects at most one set from each collection $\{Y_n^k\}_{n=1}^{N_k}$.
 Thus, $\dim_NX\leq \lfloor \alpha \rfloor \le \dim_AX$.
\qed\\

The previous proof generalizes to the case of local dimensions of metric spaces that admit well-ordered separated nets.
For example, this is the case in separable metric spaces. 
Nevertheless, if we assume the Axiom of Choice, any net can be well-ordered.

\begin{theorem}\label{thm:nagaasso2}
 Let $X$ be a metric space. Let $\alpha \ge 0$, {$C > 0$} and $R > 0$. Assume that, for all $x \in X$, we have
 {$\dim_A(B(x,R),C,R) \le \alpha$}. Then $\dim_{LC}X \le \alpha$.
\end{theorem}
\begin{proof}
 The strategy of the proof is very similar to the one of Theorem~\ref{thm:nagaasso}. Hence we only explain the differences.
Instead of having a countable maximal $\frac{R}4$-separated net of points  $\{x_i\}_i$, we now have just a maximal
$\frac{R}4$-separated net $\{x_j\}_{j\in J}$ where $J$ is a general set of indices.

By the Well-Ordering Theorem the set $J$ admits a total order such that every nonempty subset of $J$ has a least element for this
ordering. (This result is also known as Zermelo's theorem. It follows easily from Zorn's lemma and it is actually equivalent to it.)
For all $j \in J$, we select an annulus $A_{j,k_j}$ as in the  proof of Theorem~\ref{thm:nagaasso}. Then we set 
$X := \bigcup_{j \in J}A_{j,k_j}$ and
\[
 Y_j^1 := B(x_j, R-(k_j+1)r) \setminus \bigg(X_1\cup\bigcup_{i < j} B(x_i, R-k_ir)\bigg).
\]

The only nontrivial property to check is that $\{Y_j^1\}_{j \in J}$ together with $X_1$ is a cover.
Pick $x \in X$. Set
\[
 E_x := \{j \in J \,:\, x \in B(x_j,R-k_jr)\}.
\]
The set $E_x$ is nonempty, since $R-k_jr > \frac{R}2$. By the well-ordering, we have a least element $j_x \in E_x$.
Thus $x$ is in $B(x_{j_x}, R-k_{j_x}r)$ but not in any of the $B(x_i,R-k_ir)$ for $i < j_x$.
If $x$ is not in $X_1$ (and hence not in $A_{j_x,k_{j_x}}$),
we have that $x$ is in $Y_{j_x}^1$.

The rest of the proof is exactly the same.
\end{proof}

%\begin{thebibliography}{0}
%
%     
%     
% 
% 
%
%     
%\end{thebibliography}

\bibliography{general_bibliography}

\def\cprime{$'$} \def\cprime{$'$}
\providecommand{\bysame}{\leavevmode\hbox to3em{\hrulefill}\thinspace}
\providecommand{\MR}{\relax\ifhmode\unskip\space\fi MR }
\providecommand{\MRhref}[2]{%
  \href{http://www.ams.org/mathscinet-getitem?mr=#1}{#2}
}
\providecommand{\href}[2]{#2}
\begin{thebibliography}{DHLT11}

\bibitem[Ass79]{Assouad79}
Patrice Assouad, \emph{\'{E}tude d'une dimension m\'etrique li\'ee \`a la
  possibilit\'e de plongements dans {${\bf R}^{n}$}}, C. R. Acad. Sci. Paris
  S\'er. A-B \textbf{288} (1979), no.~15, A731--A734.

\bibitem[Ass82]{Assouad82}
\bysame, \emph{Sur la distance de {N}agata}, C. R. Acad. Sci. Paris S\'er. I
  Math. \textbf{294} (1982), no.~1, 31--34.

\bibitem[Ass83]{Assouad83}
\bysame, \emph{Plongements lipschitziens dans {${\bf R}^{n}$}}, Bull. Soc.
  Math. France \textbf{111} (1983), no.~4, 429--448.

\bibitem[BBI01]{Burago:book}
Dmitri{\u\i} Burago, Yuri{\u\i} Burago, and Sergei{\u\i} Ivanov, \emph{A course
  in metric geometry}, Graduate Studies in Mathematics, vol.~33, American
  Mathematical Society, Providence, RI, 2001.

\bibitem[BDS07]{Buyalo-Dranishnikov-Schroeder}
Sergei Buyalo, Alexander Dranishnikov, and Viktor Schroeder, \emph{Embedding of
  hyperbolic groups into products of binary trees}, Invent. Math. \textbf{169}
  (2007), no.~1, 153--192.

\bibitem[Bel96]{bellaiche}
Andr{\'e} Bella{\"{\i}}che, \emph{The tangent space in sub-{R}iemannian
  geometry}, Sub-Riemannian geometry, Progr. Math., vol. 144, Birkh\"auser,
  Basel, 1996, pp.~1--78.

\bibitem[BS07]{Buyalo_Schroeder_book}
Sergei Buyalo and Viktor Schroeder, \emph{Elements of asymptotic geometry}, EMS
  Monographs in Mathematics, European Mathematical Society (EMS), Z\"urich,
  2007. \MR{2327160 (2009a:53068)}

\bibitem[CC97]{Cheeger-Colding}
Jeff Cheeger and Tobias~H. Colding, \emph{On the structure of spaces with
  {R}icci curvature bounded below. {I}}, J. Differential Geom. \textbf{46}
  (1997), no.~3, 406--480.

\bibitem[DHLT11]{DeJarnette-Hajlasz-Lukyanenko-Tyson}
Noel DeJarnette, Piotr Haj{\l}asz, Anton Lukyanenko, and Jeremy Tyson, \emph{On
  the lack of density of {L}ipschitz mappings in {S}obolev spaces with
  {H}eisenberg target}, Preprint (2011).

\bibitem[DT99]{David-Toro}
Guy David and Tatiana Toro, \emph{Reifenberg flat metric spaces, snowballs, and
  embeddings}, Math. Ann. \textbf{315} (1999), no.~4, 641--710.

\bibitem[Gro81]{Gromov-polygrowth}
Mikhael Gromov, \emph{Groups of polynomial growth and expanding maps}, Inst.
  Hautes \'Etudes Sci. Publ. Math. (1981), no.~53, 53--73.

\bibitem[Hei01]{Heinonenbook}
Juha Heinonen, \emph{Lectures on analysis on metric spaces}, Universitext,
  Springer-Verlag, New York, 2001.

\bibitem[Her11]{Herron_tan}
David~A. Herron, \emph{Uniform metric spaces, annular quasiconvexity and
  pointed tangent spaces}, Math. Scand. \textbf{108} (2011), no.~1, 115--145.

\bibitem[HH00]{Hanson-Heinonen}
Bruce Hanson and Juha Heinonen, \emph{An {$n$}-dimensional space that admits a
  {P}oincar\'e inequality but has no manifold points}, Proc. Amer. Math. Soc.
  \textbf{128} (2000), no.~11, 3379--3390.

\bibitem[HS13]{Hajlasz-Schikorra}
Piotr Haj{\l}asz and Armin Schikorra, \emph{Lipschitz homotopy and density of
  {L}ipschitz mappings in {S}obolev spaces}, Accepted in Ann. Acad. Sci. Fenn.
  Math. (2013).

\bibitem[KL04]{Keith-Laakso}
Stephen Keith and Tomi Laakso, \emph{Conformal {A}ssouad dimension and
  modulus}, Geom. Funct. Anal. \textbf{14} (2004), no.~6, 1278--1321.

\bibitem[KRS12]{KaenmakiRajalaSuomala2012}
Antti K{\"a}enm{\"a}ki, Tapio Rajala, and Ville Suomala, \emph{Existence of
  doubling measures via generalised nested cubes}, Proc. Amer. Math. Soc.
  \textbf{140} (2012), no.~9, 3275--3281.

\bibitem[LD11]{LeDonne6}
Enrico Le~Donne, \emph{Metric spaces with unique tangents}, Ann. Acad. Sci.
  Fenn. Math. \textbf{36} (2011), no.~2, 683--694.

\bibitem[LD13]{LeDonne_characterization}
\bysame, \emph{A metric characterization of {C}arnot groups}, Accepted in Proc.
  Amer. Math. Soc. (2013).

\bibitem[LS98]{LuukkainenSaksman}
Jouni Luukkainen and Eero Saksman, \emph{Every complete doubling metric space
  carries a doubling measure}, Proc. Amer. Math. Soc. \textbf{126} (1998),
  no.~2, 531--534.

\bibitem[LS05]{Lang-Schlichenmaier}
Urs Lang and Thilo Schlichenmaier, \emph{Nagata dimension, quasisymmetric
  embeddings, and {L}ipschitz extensions}, Int. Math. Res. Not. (2005), no.~58,
  3625--3655.

\bibitem[Luu98]{Luukkainen}
Jouni Luukkainen, \emph{Assouad dimension: antifractal metrization, porous
  sets, and homogeneous measures}, J. Korean Math. Soc. \textbf{35} (1998),
  no.~1, 23--76.

\bibitem[Mac11]{Mackay}
John~M. Mackay, \emph{Assouad dimension of self-affine carpets}, Conform. Geom.
  Dyn. \textbf{15} (2011), 177--187.

\bibitem[Mey11]{Meyerson}
William Meyerson, \emph{The grushin plane and quasiconformal {J}acobians},
  Preprint on arXiv.org (2011).

\bibitem[MT10]{Mackay_Tyson}
John~M. Mackay and Jeremy~T. Tyson, \emph{Conformal dimension}, University
  Lecture Series, vol.~54, American Mathematical Society, Providence, RI, 2010,
  Theory and application.

\bibitem[Nag58]{Nagata58}
Jun-iti Nagata, \emph{Note on dimension theory for metric spaces}, Fund. Math.
  \textbf{45} (1958), 143--181.

\bibitem[NSW85]{nagelstwe}
Alexander Nagel, Elias~M. Stein, and Stephen Wainger, \emph{Balls and metrics
  defined by vector fields. {I}. {B}asic properties}, Acta Math. \textbf{155}
  (1985), no.~1-2, 103--147.

\bibitem[VK84]{VolbergKonyagin}
Alexander~L. Vol{\cprime}berg and Sergei~V. Konyagin, \emph{A homogeneous
  measure exists on any compactum in {${\bf R}^n$}}, Dokl. Akad. Nauk SSSR
  \textbf{278} (1984), no.~4, 783--786.

\bibitem[WY10]{Wenger-Young}
Stefan Wenger and Robert Young, \emph{Lipschitz extensions into jet space
  {C}arnot groups}, Math. Res. Lett. \textbf{17} (2010), no.~6, 1137--1149.

\end{thebibliography}
	\bibliographystyle{amsalpha}

\end{document}